\documentclass[12pt, reqno]{amsart}
\usepackage{amsmath,amssymb,amsthm}

\usepackage{amsmath, amssymb}
\usepackage{amsthm, amsfonts, mathrsfs}
\usepackage{mathptmx}
\usepackage{appendix}
\usepackage{fullpage}
\usepackage{amsfonts,graphicx}
\numberwithin{equation}{section}
\usepackage[colorlinks=true, pdfstartview=FitV, linkcolor=blue, citecolor=blue, urlcolor=blue]{hyperref}
\usepackage[sort,space,noadjust]{cite}

\def\vta{ \vartheta }
 \def\ta{ \theta }
\def\T{ \mathbb{T} }
\newcommand{\q}{{\mathbb{Q}}}

\newcommand{\R}{{\mathbb{R}}}
\def\div{ \hbox{\rm div}\,  }
\def\u{ \mathbf{u} }
\def\v{ \mathbf{v} }
\def\G{ \mathbf{G} }

\def\T{ \mathbb{T} }

\def\la{ \Lambda }
\def\aa{ \varphi }
\def\nn{\nonumber}

\newcommand{\norm}[2]{\left\lVert #1 \right\rVert_{#2}}

\theoremstyle{plain}

\newtheorem{theorem}{Theorem}[section]
\newtheorem{lemma}{Lemma}[section]

\numberwithin{equation}{section}

\begin{document}
\title[Viscous and non-resistive MHD equations]{Stability  for the  $2\frac12$-D compressible  viscous non-resistive and heat-conducting magnetohydrodynamic  flow}

	

	\author[X. Zhai]{Xiaoping Zhai}
	\address[X. Zhai]{School of Mathematics and Statistics, Guangdong University of Technology, Guangzhou, 510520, China}
	\email{pingxiaozhai@163.com}

\author[Y. Li]{Yongsheng Li}
\address[Y. Li]{School of Mathematics,
South China University of Technology,
Guangzhou, 510640, China}
\email{yshli@scut.edu.cn}

\author[Y. Zhao]{Yajuan Zhao}
\address[Y. Zhao]{Henan Academy of Big Data, ZhengZhou University,
Zhengzhou, 450001, China}
\email{zhaoyj\_91@zzu.edu.cn  (Corresponding author)}
	\begin{abstract}
	In this paper, we are concerned with the initial boundary values problem
associated to the    compressible  viscous non-resistive and heat-conducting magnetohydrodynamic  flow, where the magnetic field is vertical. More precisely, by exploiting the intrinsic structure of the system and introducing several new unknown quantities, we overcome the difficulty stemming from the lack of dissipation for  density and magnetic field, and prove the global well-posedness of strong solutions  in the framework of  Soboles spaces $H^3$. In addition, we also  get the exponential decay for this non-resistive system.
Different from the known results \cite{SIAM2021}, \cite{LXZ2023}, \cite{zhangshunhang},  we donot need the assumption that the background  magnetic field is positive here.
	\end{abstract}
	\maketitle
	
	\section{Introduction and main result}
\subsection {Model and synopsis of result}
The small data global well-posedness problem on the three dimensional (3D) non-isentropic compressible
viscous non-resistive magnetohydrodynamic (MHD) equations remains a challenging open
problem. Mathematically the concerned MHD equations are given by
\begin{eqnarray}\label{mm1}
    	\left\{\begin{aligned}
    		&\partial_t\rho+\div(\rho \v)=0,\\
    		&\rho(\partial_t\mathbf{v} + \v\cdot \nabla \v) -  \mu \Delta \v -(\lambda+\mu) \nabla \div \v
              +\nabla P=  (\nabla\times \mathbf{B})\times \mathbf{B},\\
    		&\rho C_\nu(\partial_t\vartheta+\v\cdot\nabla\vartheta)
                + P\div \v-\kappa\Delta\vartheta=2\mu |D(\v)|^2+\lambda(\div \v)^2 ,\\
            &\partial_t\mathbf{B}-\nabla\times(\v\times\mathbf{B})=0 ,\\
            &\div\mathbf{B}=0,\\
            &(\rho,\v,\vartheta,\mathbf{B})|_{t=0}=(\rho_0,\v_0,\vartheta_0,\mathbf{B}_0).
    	\end{aligned}\right.
    \end{eqnarray}
    Here,  $\rho$, $\v$, $\vartheta$, $\mathbf{B}$
    denotes the density, velocity field, temperature
    and magnetic field of the MHD fluids, respectively.
    The parameters $\mu$  and $\lambda$ are shear viscosity and volume viscosity coefficients,  respectively,
    which satisfy the standard  strong parabolicity assumption,
    \begin{align*}
    	\mu>0\quad\hbox{and}\quad
    	\nu:=\lambda+2\mu>0.
    \end{align*}
$C_\nu$ is a positive constant
    and $\kappa > 0$ is the heat-conductivity coefficient.
The fluid is assumed to obey the ideal polytropic law,
  so the pressure $P=R\rho\vartheta$ for a positive constant $R$.
The deformation tensor  $D(\v) = \frac{1}{2} \big( \nabla \v + (\nabla \v)^{T} \big)$.
The compressible MHD equations model the motion of electrically conducting fluids in the presence of
    a magnetic field.
The compressible MHD  equations  can be derived from the isentropic Navier-Stokes-Maxwell system by taking
    the zero dielectric constant limit \cite{lifucai}.

When the effect of the magnetic field can be neglected or $\mathbf{B} = 0$,
       \eqref{mm1} reduces to the compressible Navier-Stokes equations.
Some early works on the global existence theory and uniqueness of
    the initial value problem of compressible Navier-Stokes equations
    when the solution is perturbed near the equilibrium state
    were obtained by Matsumura and Nishida \cite{MN1979,MN1980,MN1983}.
Moreover, Matsumura and Nishida \cite{MN1979} obtained that
  the time-decay rate of the solution in the $H^2$ norm is $-\frac{3}{4}$ provided that
  the initial perturbation is sufficiently small in $H^4\cap L^1$.
Then the global well-posedness and time-decay estimates
  for this equations have been studied by many researchers,
  see \cite{GW2021,HLX2012,JWX2018,LX2019,XZ2021,ZC2020} and the references therein.
Recently, Wang and Wen \cite{WW2022} investigated the full compressible Navier-Stokes equations with reaction
  diffusion, and gave the results of global well-posedness and some optimal decay estimates of the solutions in the whole space.

When $\rho=\vartheta=$const, \eqref{mm1} becomes the incompressible MHD system without magnetic diffusion.
  For the incompressible MHD equations with full dissipation or partial dissipation,
  the global well-posedness and many other properties have been studied,
  see \cite{AZ2017,CL2018,CZZ2022,JYZ2017,PZZ2018,ST1983,WZ2021,SH2012} and the references therein.
When $\vartheta=$const, Hu and Wang \cite{HW2010} studied the global existence
   and large-time behavior of solutions to MHD equations \eqref{mm1} with full dissipation
   in a bounded domain $\Omega\subset\R^3$ with initial-boundary conditions.
And then the system \eqref{mm1} with full dissipation for the isentropic case has been studied
  in many papers, for examples,
   \cite{DF2006,WZZ2021,WWZ2022} and soon on.
Recently, Wu, Wang and Zhang \cite{WWZ2022} obtained  the upper optimal decay rates for higher
     order spatial derivatives of the solution:
$$\|\nabla(\rho-1)\|_{H^1}+\|\nabla \v\|_{H^1}
     +\|\nabla \mathbf{B}\|_{H^1}\lesssim(1+t)^{-\frac{3}{4}(\frac{2}{p}-1)-\frac{1}{2}},$$
where $p\in[1,2]$.
For the non-resistive compressible MHD equations,
    Li and Sun \cite{JDE2019} obtained the existence of global weak solutions for
    the motion of fluids takes place in a bounded regular domain $\Omega\subset\R^2$.
Very recently, Dong, Wu and Zhai \cite{DWZ2023}  considered
     the $2\frac{1}{2}$-D compressible viscous non-resistive MHD equations
     and established the upper optimal decay rates of the solution.

For the  non-isentropic case,  the full compressible MHD equations
    have attracted the interests of many physicists and mathematicians
    (see, e.g., \cite{CW2002,PG2013,HW2008,ST2022,xin1,xin2} and the references therein).
Recently, by using {\itshape a  priori} estimates and  the continuity argument,
   Huang and Fu \cite{HF2023} obtain the global-in-time solution of the 3D non-isentropic MHD system,
   where the initial data should be bounded in $L^2$ and is small in $H^2$.
   At the same time, they   got the optimal decay estimates of the highest-order derivatives:
   $$\|\nabla^k(\rho-1,\v,\vartheta-1,\mathbf{B})\|_{L^2(\R^3)}
      \lesssim (1+t)^{-\frac{3}{4}-\frac{k}{2}},\qquad k=0,1,2.$$
By assuming that the motion of fluids takes place in the plane
   while the magnetic field acts on the fluids only in the vertical direction,
    Li and Sun \cite{SIAM2021} obtained the existence of global weak solutions
    with large initial data for the non-resistive MHD system.
Then, Li \cite{JMFM2022} proved the global well-posedness of
      the three-dimensional full compressible viscous
      non-resistive MHD system in an infinite slab $\R^2\times(0,1)$
       with strong background magnetic field.
Li, Xu and Zhai \cite{LXZ2023} obtained the global small solutions of the MHD equations
    without magnetic diffusion in $\T^3$ when
     the initial magnetic field is close to a background magnetic field satisfying the Diophantine condition  much lately.
This paper focuses on the small data  global well-posedness problem
for a very special $2\frac12$-D  compressible non-resistive MHD system.
    The motion of fluids takes place in the plane
    while the magnetic field acts on fluids only in the vertical direction, namely
    \begin{align*}
    &\rho\stackrel{\mathrm{def}}{=}\rho(t,x_1,x_2),\quad\v=(\v^1(t,x_1,x_2),\v^2(t,x_1,x_2),0)\stackrel{\mathrm{def}}{=}(\u,0),\nn\\
    &\vartheta\stackrel{\mathrm{def}}{=}\vartheta(t,x_1,x_2),
    \quad\mathbf{B}\stackrel{\mathrm{def}}{=}(0,0,m(t,x_1,x_2)).
    \end{align*}
    Then \eqref{mm1} is reduced to
    \begin{eqnarray}\label{mm2}
    \left\{\begin{aligned}
    &\partial_t\rho+\div(\rho \u)=0,\\
    &\rho(\partial_t\mathbf{u} + \u\cdot \nabla \u) -  \mu \Delta \u -(\lambda+\mu) \nabla \div \u+\nabla P +\frac12\nabla m^2=  0,\\
    &\rho C_\nu(\partial_t\vartheta+\u\cdot\nabla\vartheta)
    + P\div \u-\kappa\Delta\vartheta=2\mu |D(\u)|^2+\lambda(\div \u)^2 ,\\
    &\partial_tm+\div(m \u)=0,\\
      &(\rho,\u,\vartheta,m)|_{t=0}=(\rho_0,\u_0,\vartheta_0,m_0).
    \end{aligned}\right.
    \end{eqnarray}
    Clearly $(\rho^{(0)}, \u^{(0)}, \vartheta^{(0)},m^{(0)})$ with
    $$
    \rho^{(0)}=1, \quad \u^{(0)}={\mathbf{0}},  \quad \vartheta^{(0)}={1}, \quad  m^{(0)}=0
    $$
    solves (\ref{mm2}).

\subsection {Main result}
From now on, attention is focused on the system \eqref{mm2}  for
$(t, x) \in  [0, \infty)\times\T^2$ with the volume of $\T^2$ normalized to unity,
$$|\T^2|=1.$$
For notational convenience, we write
\begin{align}\label{conserint0}
\int_{\T^2}\rho_0\,dx=\bar{\rho},\quad
\int_{\T^2}\rho_0\u_0\,dx=0.
\end{align}
In addition, we assume, for some positive constant $e_0$, that
\begin{align}\label{conserint1}
\int_{\mathbb T^2} \rho_0 \vta_0 \,dx+\frac12\int_{\mathbb T^2} \rho_0
|\u_0|^2dx+\frac12\int_{\mathbb T^2} |m_0|^2\,dx=e_0.
\end{align}
Owing to the conservation of total mass,
total momentum, and total energy,  we
have, for any $t\ge0,$ that
\begin{align}
	&\int_{\mathbb T^2} \rho(x) \,dx =\bar{\rho},
\quad\int_{\mathbb T^2}\rho(x) \u(x) \,dx =0,\label{conserint2}\\
	&\int_{\mathbb T^2} \rho \vta \,dx+\frac12\int_{\mathbb T^2} \rho
	|\u|^2dx+\frac12\int_{\mathbb T^2} |m|^2\,dx=e_0.\label{conserint3}
\end{align}
The main result of the paper is stated as follows.
\begin{theorem}\label{dingli}
 Assume that the initial data $(\rho_0,\u_0,\vartheta_0, m_0)$ satisfy \eqref{conserint0}, \eqref{conserint1} and
 $(\rho_0-\bar{\rho},\u_0,\vartheta_0-\bar{\vartheta}, m_0)\in H^{3}(\T^2)$ with
\begin{align}\label{pingjunwei0}
	 c_0\le\rho_0,\vartheta_0\le c_0^{-1}
\end{align}
for some constant $c_0>0$.
 If there exists a small constant $\varepsilon$ such that
\begin{align*}
\norm{(\rho_0-\bar{\rho},\u_0,\vartheta_0-\bar{\vartheta},m_0)}{H^{3}}\le\varepsilon,
\end{align*}
then the system \eqref{mm2} admits a unique  global  solution
$(\rho-\bar{\rho},\u,\vartheta-\bar{\vartheta},m)$ such that
\begin{align*}
&(\rho-\bar{\rho},m)\in C(\R^+;\ H^{3}),\quad
(\u,\vartheta-\bar{\vartheta}) \in C(\R^+;\ H^{3})\cap L^{2} (\R^+;\ H^{4}).
\end{align*}
Moreover, for any $t\ge 0$,
there holds
\begin{align*}
\norm{(\u,\vartheta-\bar{\vartheta})}{H^{{3}}}\le C_1 e^{-C_1t},
\end{align*}
for some pure constant $C_1 >0$.
\end{theorem}

\bigskip

\vskip .3in
\section{Preliminaries}

This section provides  several functional inequalities to be used in the proof of our main result.
We first recall a weighted Poincar\'e inequality first established by Desvillettes and Villani in \cite{DV05}.
\begin{lemma}\label{lem-Poi}
	Let $\Omega$  be a bounded connected Lipschitz domain and $\bar{\varrho}$ be a positive constant.  There exists a positive constant $C$, depending on $\Omega$ and $\bar{\varrho}$, such that  for any nonnegative function $\varrho$ satisfying
	\begin{align*}
		\int_{\Omega}\varrho dx=1, \quad \varrho\le\bar{\varrho},
	\end{align*}
	and any $\u\in H^1(\Omega)$, there holds
	\begin{align}\label{wPoi}
		\int_{\Omega}\varrho \left(\u-\int_{\Omega}\varrho \u \,dx\right)^2\,dx\le C\|\nabla \u\|_{L^2}^2.
	\end{align}
\end{lemma}
\vskip .1in
In order to remove the weight function $\varrho$ in  \eqref{wPoi} without resorting to the lower bound of  $\varrho$, we need  another variant of Poincar\'e inequality (see Lemma 3.2 in \cite{F04}).
\begin{lemma}\label{lem2.2}
	Let $\Omega$ be a bounded connected Lipschitz domain  and
	$p>1$ be a constant. Given positive constants $M_0$ and $E_0$, there is a constant $C=C(E_0,M_0)$ such that for any
	non-negative function $\varrho$ satisfying
	$$
	M_0\leq\int_{\Omega}\varrho dx\quad\mbox{and}\quad  \int_{\Omega}\varrho^{p}dx\leq
	E_0,
	$$
	and for any $\u\in H^1(\Omega)$, there holds
	$$
	\|\u\|_{L^2}^2\leq C\left[\|\nabla
	\u\|_{L^2}^2+\left(\int_{\Omega}\varrho|\u|\,dx\right)^2\right].
	$$
\end{lemma}
\vskip .1in
\begin{lemma}\label{daishu}{\rm(\cite{kato})}
Let $s\ge 0$, $f,g\in {H^{s}}(\T^2)\cap {L^\infty}(\T^2)$, it holds that
\begin{equation*}
\|fg\|_{H^{s}}\le C(\|f\|_{L^\infty}\|g\|_{H^{s}}+\|g\|_{L^\infty}\|f\|_{H^{s}}).
\end{equation*}
\end{lemma}
\vskip .1in
\begin{lemma}\label{jiaohuanzi}{\rm(\cite{kato})}
Let $s> 0$. Then there exists a constant $C$ such that, for any $f\in {H^{s}}(\T^2)\cap W^{1,\infty}(\T^2)$, $g\in {H^{s-1}}(\T^2)\cap {L^\infty}(\T^2)$, there holds
\begin{align*}
\norm{[\la^s,f\cdot\nabla ]g}{L^2}\le C(\norm{\nabla f}{L^\infty}\norm{\la^sg}{L^2}+\norm{\la^s f}{L^2}\norm{\nabla g}{L^\infty}).
\end{align*}
\end{lemma}
\vskip .1in
\begin{lemma}\label{fuhe}{\rm(\cite{Triebel})}
Let $s>0$ and $f\in H^s(\T^2)\cap L^\infty(\T^2)$. Assume that $F$ is a smooth function  on $\R$ with $F(0)=0$. Then we
have
$$
\|F(f)\|_{H^s}\le C(1+\|f\|_{L^\infty})^{[s]+1}\|f\|_{H^s},
$$
where the constant $C$ depends on
$\sup\limits_{k\le{[s]+2},\, t\le\|f\|_{L^\infty}} \left\|F^{(k)}(t)\right\|_{L^\infty}.$
\end{lemma}
At the last part of  this subsection,
we are  devoted to proving the following Poincar\'e type inequality for $\ta=\vartheta-1$.
\begin{lemma}\label{poincareineq}
	Let $(\rho,\u, \vta, m)$ be smooth solutions to \eqref{mm2} satisfying
	\eqref{conserint2}, \eqref{conserint3} and
	\begin{align}\label{jinq0}
		c_0\le\rho, {\vta}\le c_0^{-1},
	\end{align}
	then we have the following estimate
	\begin{align}\label{jinq1}
		\|\ta\|_{L^2}^2\leq& \,
		C\|\nabla\ta\|_{L^2}^2+C\|\nabla\u\|_{L^2}^4+C\|\aa\|_{L^2}^2,
	\end{align}
where $\ta={\vta}-1$
\end{lemma}
\begin{proof}
 Without loss of generality, we set $ e_0=1$ in \eqref{conserint3}.
   From \eqref{conserint2} and \eqref{conserint3}, one has
\begin{align}
	&\int_{\mathbb T^2} \rho(x) \,dx =1,\quad\int_{\mathbb T^2} \rho(x) \u(x)
	\,dx=0,\label{conser2}\\
	&\int_{\mathbb T^2} \rho \vta \,dx+\frac12\int_{\mathbb T^2} \rho
	|\u|^2dx+\frac12\int_{\mathbb T^2} |m|^2\,dx=1.\label{conser3}
\end{align}
Now, it's easy to deduce from \eqref{jinq0} that
	\begin{align}\label{jinq3}
		\|\ta\|_{L^2}^2
		\le&\,C\|\sqrt\rho\ta\|_{L^2}^2\nn\\
		=&\,C\int_{\mathbb T^2} \rho \Big|\ta-\int_{\mathbb T^2} \rho \ta
		\,dx+\int_{\mathbb T^2} \rho \ta \,dx\Big|^2 \,dx\nn\\
		\le&\, C\int_{\mathbb T^2} \rho \Big|\ta-\int_{\mathbb T^2} \rho \ta
		\,dx\Big|^2 \,dx
		+C\int_{\mathbb T^2} \rho \Big|\int_{\mathbb T^2} \rho \ta \,dx\Big|^2
		\,dx.
	\end{align}
By Lemma \ref{lem-Poi},
	\begin{align}\label{jinq4}
		\int_{\mathbb T^2} \rho \Big|\ta-\int_{\mathbb T^2} \rho \ta
		\,dx\Big|^2 \,dx\le C\|\nabla\ta\|_{L^2}^2.
	\end{align}
 Thanks to \eqref{conser3}, we have
	\begin{align*}
		\int_{\mathbb T^2} \rho \ta \,dx=&\int_{\mathbb T^2} \rho (\vta-1) \,dx\nn\\
=&-\frac12\int_{\mathbb T^2} \rho
		|\u|^2dx-\int_{\mathbb T^2} \rho
		dx-\frac12\int_{\mathbb T^2} |m|^2\,dx+1\nn\\
=&-\frac12\int_{\mathbb T^2} \rho
		|\u|^2dx-\int_{\mathbb T^2} \aa\,dx.
	\end{align*}
Then the last  term in \eqref{jinq3} can be bounded as follows,
	\begin{align}\label{jinq5}
		\int_{\mathbb T^2} \rho \Big|\int_{\mathbb T^2} \rho \ta \,dx\Big|^2
		\,dx
		=& \Big|\int_{\mathbb T^2} \rho \ta
		\,dx\Big|^2\leq\Big|\frac12\|\sqrt\rho\u\|_{L^2}^2+\|\aa\|_{L^2}\Big|^2\nn\\
		\le&C\left(\|\u\|_{L^2}^4+\|\aa\|_{L^2}^2\right).
	\end{align}
Due to $\int_{\mathbb{T}^2}\rho \u \,dx=0,$ one can deduce from Lemma \ref{lem-Poi} that
\begin{align*}
\|(\sqrt{\rho }\u)(t)\|_{L^2}^2\le C\|\nabla \u(t)\|_{L^2}^2,
\end{align*}
which combines with Lemma \ref{lem2.2} furhter implies that
\begin{align}\label{P}
\|\u(t)\|_{L^2}^2\le C\|\nabla \u(t)\|_{L^2}^2.
\end{align}
From \eqref{jinq5}  and \eqref{P}, we get
	\begin{align}\label{jinq8}
		\int_{\mathbb T^2} \rho \Big|\int_{\mathbb T^2} \rho \ta \,dx\Big|^2
		\,dx
		\le C\|\nabla\u\|_{L^2}^4+ C\|\aa\|_{L^2}^2.
	\end{align}
Inserting \eqref{jinq4} and \eqref{jinq8} in \eqref{jinq3}, we obtain
	\eqref{jinq1}.
This completes the proof of Lemma \ref{poincareineq}.
\end{proof}

\vskip .3in
\section{The proof of  Theorem \ref{dingli}}

This section is devoted to proving Theorem \ref{dingli}. The proof is long
and is thus divided into several
subsections for the sake of clarity.

\subsection{Local well-posedness}\label{loc}
Given the initial data $(\rho_0-\bar{\rho},\u_0,\vartheta_0-\bar{\vartheta}, m_0)\in H^{3}(\T^2)$,
 the local well-posedness of \eqref{mm2} could be proven
by using the standard energy method (see, e.g.,  \cite{Kawashima}). Thus, we may assume that there exists $T > 0$  such that the system \eqref{mm2} has a unique solution
$(\rho-\bar{\rho},\u,\vartheta-\bar{\vartheta},m)\in C([0,T];H^{3})$. Moreover,
\begin{align}\label{youjiexing}
\frac12c_0\le\rho(t,x),\vartheta(t,x)\le 2c_0^{-1},\quad\hbox{for any $t\in[0,T]$}.
\end{align}
We use the bootstrapping argument to show that this local solution can be extended into a global one.
The goal is to derive {\itshape a priori} upper bound. To initiate the bootstrapping argument, we make the
ansatz that
\begin{align*}
\sup_{t\in[0,T]}(\norm{\rho-\bar{\rho}}{H^{3}}+\norm{\u}{H^{3}}
     +\norm{\vartheta-\bar{\vartheta}}{H^{3}}+\norm{m}{H^{3}})\le \delta,
\end{align*}
where $0<\delta<1$ obeys requirements to be specified later. In the following subsections we prove that, if the initial norm is taken to be sufficiently small, namely
$$
\|(\rho_0-\bar{\rho}, \u_0,\vartheta_0-\bar{\vartheta}, m_0)\|_{H^{3}} \le \varepsilon,
$$
with sufficiently small $\varepsilon>0$, then
$$
\sup_{t\in[0,T]}(\norm{\rho-\bar{\rho}}{H^{3}}+\norm{\u}{H^{3}}
     +\norm{\vartheta-\bar{\vartheta}}{H^{3}}+\norm{m}{H^{3}})\le \frac{\delta}{2}.
$$
The bootstrapping argument then leads to the desired global bound.

\subsection{Energy estimates for $(\rho-\bar \rho,\u,\vartheta-\bar \vartheta,m)$}
We first show the energy estimates
which contains the dissipation estimate for $\u$ and $\vartheta-\bar \vartheta$ only.
Without loss of generality,  let $\bar \rho=\bar{\vartheta}=1$
 and define
\begin{align*}
a\stackrel{\mathrm{def}}{=}\rho -1,\quad
\ta\stackrel{\mathrm{def}}{=}\vartheta -1.
\end{align*}
Then, we can rewrite \eqref{mm2} into the following form
\begin{eqnarray}\label{mm3}
\left\{\begin{aligned}
&\partial_t a+\div \u+\u\cdot\nabla a+a\,\div \u=0,\\
&\partial_t \u-\div(\bar{\mu}(\rho)\nabla\u)-\nabla(\bar{\lambda}(\rho)\div\u)+\nabla\aa
+\nabla\ta=f_1,\\
&\partial_t\theta-\div(\bar{\kappa}(\rho)\nabla\ta)+ \div \u=f_2,\\
&\partial_t{m}+\u\cdot\nabla {m}+{m}\,\div \u=0,\\
&(a,\u,\ta,m)|_{t=0}=(a_0,\u_0,\ta_0,m_0),
\end{aligned}\right.
\end{eqnarray}
where
$$\aa=a +\frac12 m^2,\quad
\bar{\mu}(\rho)=\frac{\mu}{\rho},\quad
\bar{\lambda}(\rho)=\frac{\lambda+\mu}{\rho},\quad
\bar{\kappa}(\rho)=\frac{\kappa}{\rho},$$
and
\begin{align}\label{f1}
f_1=&-\u\cdot\nabla \u+I(a)\nabla \aa+\mu\nabla I(a)\nabla\u
  +(\lambda+\mu)\nabla I(a)\div\u-\ta\nabla J(a),
\end{align}
\begin{align}\label{f2}
f_2=&- \div (\ta\u)+ \kappa\nabla I(a)\nabla\ta
  +\frac{1}{1+a}\Big(2\mu |D(\u)|^2+\lambda(\div \u)^2\Big),
\end{align}
with $I(a)=\frac{a}{1+a}$ and $J(a)=\mathrm{ln}(1+a).$
\vskip .1in
In this subsection, we shall prove the following crucial lemma.
\begin{lemma}\label{gaojie}
Let $(a,\u,\ta,m) \in C([0, T];H^3)$ be a solution of \eqref{mm3}, then there holds
\begin{align}\label{gaojie1}
&\frac12\frac{d}{dt}\norm{(a,\u,\ta,m)}{H^{3}}^2
     -\frac{1}{2}\frac{d}{dt}\int_{\T^2}\frac{a}{1+a}(\Lambda^3a)^2dx
      +\mu\norm{\nabla\u}{H^{3}}^2
     +(\lambda+\mu)\norm{\div\u}{H^{3}}^2+\kappa\norm{\nabla\ta}{H^{3}}^2\nn\\
& \quad\leq C\big((1+\norm{a}{H^3}^2)(1+\norm{\u}{H^3})\norm{\u}{H^3}
       +\norm{(\aa,\u,\ta)}{H^3}^2\big)
      \|( a, \u,\ta, m)\|_{H^{3}}^2.
\end{align}
\end{lemma}
\begin{proof}
At first, we can use the standard energy method to get the $L^2$ estimates, here we omit the details. Let $s=1,2,3,$ and $\Lambda=\sqrt{-\Delta}$.
Applying  operator $\la^s$  to the equations of \eqref{mm3} and then taking $L^2$ inner product with $(\la^s a, \la^s\u,\la^s \ta,\la^s m)$ yield
\begin{align}\label{gaojie2}
&\frac12\frac{d}{dt}\norm{(\la^s a,\la^s\u,\la^s\ta,\la^sm)}{L^2}^2
-\int_{\T^2}\la^{s}\div(\bar{\mu}(\rho)\nabla\u)\cdot\la^{s} \u\,dx\nn\\
&\qquad-\int_{\T^2}\la^{s}\nabla(\bar{\lambda}(\rho)\div\u)\cdot\la^{s} \u\,dx-\int_{\T^2}\la^{s}\div(\bar{\kappa}(\rho)\nabla\ta)\cdot\la^{s} \ta\,dx\nn\\
&\quad= -\int_{\T^2}\la^{s} (\u\cdot\nabla a+ a\,\div \u)\cdot\la^s a\,dx
-\int_{\T^2}\la^{s}\div\u\cdot \la^s a\,dx-\int_{\T^2}\la^{s} \nabla\aa\cdot\la^{s} \u\,dx\nn\\
&\qquad
-\int_{\T^2}\la^{s} \nabla\ta\cdot\la^{s} \u\,dx
+\int_{\T^2}\la^{s} f_1\cdot\la^{s} \u\,dx+\int_{\T^2}\la^{s} f_2\cdot\la^{s} \ta\,dx
-\int_{\T^2}\la^{s} \div\u\cdot\la^{s} \ta\,dx\nn\\
&\qquad
-\int_{\T^2}\la^{s} (\u\cdot\nabla m+m\,\div \u)\cdot\la^{s} b\,dx-\int_{\T^2}\la^{s}\div\u\cdot \la^{s}m\,dx.
\end{align}
Due to $$
\nabla\aa=\nabla a+
m\nabla m
$$
and the  cancellations
\begin{align*}
&\int_{\T^2}\la^{s}\div\u\cdot \la^{s}a\,dx+\int_{\T^2}\la^{s}\nabla a\cdot\la^{s}\u\,dx=0,\nn\\
&\int_{\T^2}\la^{s}\div\u\cdot \la^{s}\ta\,dx+\int_{\T^2}\la^{s}\nabla \ta\cdot\la^{s}\u\,dx=0,
\end{align*}
we can further rewrite \eqref{gaojie2} into
\begin{align}\label{gaojie3}
&\frac12\frac{d}{dt}\norm{(\la^s a,\la^s\u,\la^s\ta,\la^sm)}{L^2}^2
-\int_{\T^2}\la^{s}\div(\bar{\mu}(\rho)\nabla\u)\cdot\la^{s} \u\,dx\nn\\
&\qquad-\int_{\T^2}\la^{s}\nabla(\bar{\lambda}(\rho)\div\u)\cdot\la^{s} \u\,dx-\int_{\T^2}\la^{s}\div(\bar{\kappa}(\rho)\nabla\ta)\cdot\la^{s} \ta\,dx\nn\\
&\quad= -\int_{\T^2}\la^{s} (\u\cdot\nabla a+ a\,\div \u)\cdot\la^s a\,dx
-\int_{\T^2}\la^{s} (m\nabla m)\cdot\la^{s} \u\,dx\nn\\
&\qquad
+\int_{\T^2}\la^{s} f_1\cdot\la^{s} \u\,dx+\int_{\T^2}\la^{s} f_2\cdot\la^{s} \ta\,dx-\int_{\T^2}\la^{s} (\u\cdot\nabla m+m\,\div \u)\cdot\la^{s} m\,dx.
\end{align}
The second term on the left-hand side can be written as
\begin{align}\label{gaojie4}
&-\int_{\T^2}\la^{s}\div(\bar{\mu}(\rho)\nabla\u)\cdot\la^{s} \u\,dx\nn\\
&\quad
=\int_{\T^2}\la^{s}(\bar{\mu}(\rho)\nabla\u)\cdot\nabla\la^{s} \u\,dx\nn\\
&\quad=\int_{\T^2}[\la^{s},\,\bar{\mu}(\rho)\nabla]\u\cdot\nabla\la^{s} \u\,dx
        +\int_{\T^2}\bar{\mu}(\rho)\nabla\la^{s}\u\cdot\nabla\la^{s} \u\,dx.
\end{align}
Due to \eqref{youjiexing}, we have for any $t\in[0,T]$ that
\begin{align}\label{gaojie5}
 \int_{\T^2}\bar{\mu}(\rho)\nabla\la^{s}\u\cdot\nabla\la^{s} \u\,dx
 \ge \frac{\mu c_0}{2}\norm{\la^{s+1}\u}{L^{2}}^2.
\end{align}
For the first term on the right-hand side in \eqref{gaojie4}, we first rewrite this term into
\begin{align}\label{gaojie6}
\int_{\T^2}[\la^{s},\bar{\mu}(\rho)\nabla]\u\cdot\nabla\la^{s} \u\,dx
=&\int_{\T^2}[\la^{s},\,(\bar{\mu}(\rho)-\mu+\mu)\nabla]\u\cdot\nabla\la^{s} \u\,dx\nn\\
=&-\int_{\T^2}[\la^{s},\,\mu I(a)\nabla]\u\cdot\nabla\la^{s} \u\,dx.
\end{align}
Bounding nonlinear terms involving composition functions in $\eqref{gaojie6}$ is more elaborate.
Throughout we make the assumption that
\begin{equation}\label{axiao}
\sup_{t\in\R^+,\, x\in\T^2} |a(t,x)|\leq \frac12
\end{equation}
which will enable us to use freely the composition estimate stated in Lemma \ref{fuhe}.
Note that $H^3(\T^2)\hookrightarrow L^\infty(\T^2),$ condition \eqref{axiao} will be ensured by the fact that the constructed solution about $a$ has small norm.
It follows from Lemma \ref{fuhe} that
\begin{equation}\label{eq:smalla}
\|I(a)\|_{H^s}\le C\|a\|_{H^s}\quad\hbox{and}\quad \|J(a)\|_{H^s}\le C\|a\|_{H^s}, \quad\hbox{for any $s>0$}.
\end{equation}
Then, using the fact $H^2(\T^2)\hookrightarrow L^\infty(\T^2)$,
with the aid of Lemmas \ref{jiaohuanzi} and \eqref{eq:smalla}, we have
\begin{align}\label{gaojie9}
&\Big|\int_{\T^2}[\la^{s},\mu I(a)\nabla]\u\cdot\nabla\la^{s} \u\,dx\Big| \nn\\
&\quad \le C\Big(\norm{\nabla I(a) }{L^\infty}\norm{\la^{s} \u}{L^2}
                 +\norm{\nabla \u}{L^\infty}\norm{\la^{s}I(a)}{L^2}\Big)
         \norm{\nabla\la^{s}\u}{L^{2}}\nn\\
&\quad \le C\big(\norm{I(a) }{H^3}\norm{\la^{s} \u}{L^2}
           +\norm{\nabla \u}{L^\infty}\norm{I(a)}{H^3}\big)
         \norm{\nabla\la^{s}\u}{L^{2}}\nn\\
&\quad \le \frac{\mu c_0}{4}\norm{\la^{s+1}\u}{L^{2}}^2+C\norm{ a }{H^3}^2\norm{ \u}{H^3}^2.
\end{align}
Inserting \eqref{gaojie5}, \eqref{gaojie6} and \eqref{gaojie9} into \eqref{gaojie4} leads to
\begin{align}\label{gaojie10}
-\int_{\T^2}\la^{s}\div(\bar{\mu}(\rho)\nabla\u)\cdot\la^{s} \u\,dx
\ge&\frac{\mu c_0}{4}\norm{\la^{s+1}\u}{L^{2}}^2
     -C\norm{a}{H^3}^2\norm{ \u}{H^3}^2.
\end{align}
Similarly, we have
\begin{align}\label{gaojie11}
-\int_{\T^2}\la^{s}\nabla(\bar{\lambda}(\rho)\div\u)\cdot\la^{s} \u\,dx
\ge&\frac{(\lambda+\mu)c_0}{4}\norm{\la^{s}\div\u}{L^{2}}^2
     -C\norm{a}{H^3}^2\norm{ \u}{H^3}^2
\end{align}
and
\begin{align}\label{gaojie12}
-\int_{\T^2}\la^{s}\div(\bar{\kappa}(\rho)\nabla\ta)\cdot\la^{s} \ta\,dx
\ge&\frac{\kappa c_0}{4}\norm{\la^{s+1}\ta}{L^{2}}^2
  -C\norm{a}{H^3}^2\norm{ \ta}{H^3}^2.
\end{align}
Hence, plugging \eqref{gaojie10}-\eqref{gaojie12} into \eqref{gaojie3}, we obtain
\begin{align}\label{gaojie13}
&\frac12\frac{d}{dt}\norm{(\la^s a,\la^s\u,\la^s\ta,\la^sm)}{L^2}^2
+\frac{\mu c_0}{4}\norm{\la^{s+1}\u}{L^2}^2
     +\frac{(\lambda+\mu)c_0}{4}\norm{\la^{s}\mbox{div}\u}{L^2}^2
     +\frac{\kappa c_0}{4}\norm{\la^{s+1}\ta}{L^2}^2  \nn\\
&\quad\leq-\int_{\T^2}\la^{s} (\u\cdot\nabla a+ a\,\div \u)\cdot\la^s a\,dx
-\int_{\T^2}\la^{s} (m\nabla m)\cdot\la^{s} \u\,dx+\int_{\T^2}\la^{s} f_1\cdot\la^{s} \u\,dx\nn\\
&\qquad
+\int_{\T^2}\la^{s} f_2\cdot\la^{s} \ta\,dx
-\int_{\T^2}\la^{s} (\u\cdot\nabla m+m\,\div \u)\cdot\la^{s} m\,dx
+C\big(\norm{ \u}{H^3}^2+\norm{ \ta}{H^3}^2\big)\norm{a}{H^3}^2.
\end{align}
To bound the first terms on the right-hand side of \eqref{gaojie13},
we rewrite
\begin{align}\label{gaojie14}
\int_{\T^2}\la^{s} (\u\cdot\nabla a)\cdot\la^s a\,dx
=&\int_{\T^2} [\la^{s},\,\u\cdot\nabla] a\cdot\la^s a\,dx
+\int_{\T^2}\u\cdot\nabla\la^{s}a\cdot\la^s a\,dx.
\end{align}
By Lemma \ref{jiaohuanzi}, there holds
\begin{align}\label{gaojie15}
\left|\int_{\T^2}[\la^{s},\,\u\cdot\nabla] a\cdot\la^s a\,dx\right|\le& C\norm{[\la^s,\u\cdot\nabla ]a}{L^2}\norm{\la^{s}a}{L^2}\nn\\
\le&C(\norm{\nabla \u}{L^\infty}\norm{\la^{s}a}{L^2}+\norm{\la^{s} \u}{L^2}\norm{\nabla a}{L^\infty})\norm{\la^{s}a}{L^2}\nn\\
\le&C\norm{\u}{H^3}\norm{a}{H^3}^2.
\end{align}
By integration by parts, we have
\begin{align}\label{gaojie16}
\left|\int_{\T^2}\u\cdot\nabla\la^{s}a\cdot\la^s a\,dx\right|\le& C\norm{\nabla \u}{L^\infty}\norm{\la^{s}a}{L^2}^2\le C\norm{\u}{H^3}\norm{a}{H^3}^2.
\end{align}
From \eqref{gaojie14}-\eqref{gaojie16}, we get
\begin{align}\label{gaojie17}
\left|\int_{\T^2}\la^{s} (\u\cdot\nabla a)\cdot\la^s a\,dx\right|
\le&C\norm{\u}{H^3}\norm{a}{H^3}^2.
\end{align}
In the same manner, there holds
\begin{align}\label{gaojie18}
\left|\int_{\T^2}\la^{s} (\u\cdot\nabla m)\cdot\la^{s} m\,dx\right|
\le&C\norm{\u}{H^3}\norm{m}{H^3}^2.
\end{align}
Similarly,
\begin{align}\label{gaojie19}
&\int_{\T^2}\la^{s} (m\,\div \u)\cdot\la^{s} m\,dx
+\int_{\T^2}\la^{s} (m\nabla m)\cdot\la^{s} \u\,dx\nn\\
&\quad=\int_{\T^2}[\la^{s},\, m\div] \u\cdot\la^{s} m\,dx
+\int_{\T^2}[\la^{s},\,m\nabla] m\cdot\la^{s} \u\,dx\nn\\
&\qquad+\int_{\T^2}m\div\la^{s} \u\cdot\la^{s} m\,dx
+\int_{\T^2}m \nabla\la^{s} m\cdot\la^{s} \u\,dx\nn\\
&\quad=\int_{\T^2}[\la^{s},\, m\div] \u\cdot\la^{s} m\,dx
+\int_{\T^2}[\la^{s},\,m\nabla] m\cdot\la^{s} \u\,dx
+\int_{\T^2}m\div (\la^{s}m\la^{s}\u)\,dx\nn\\
&\quad=\int_{\T^2}[\la^{s}, \,m\div] \u\cdot\la^{s} m\,dx
+\int_{\T^2}[\la^{s},\,m\nabla] m\cdot\la^{s} \u\,dx
-\int_{\T^2}\la^{s}m\la^{s}\u\cdot\nabla m\,dx.
\end{align}
It follows from Lemma \ref{jiaohuanzi} that
\begin{align}\label{gaojie20}
\int_{\T^2}[\la^{s}, m\div] \u\cdot\la^{s} m\,dx
\le C\norm{\u}{H^3}\norm{m}{H^3}^2
\end{align}
and
\begin{align}\label{gaojie21}
\int_{\T^2}[\la^{s}, m\nabla] m\cdot\la^{s} \u\,dx
\le C\norm{\u}{H^3}\norm{m}{H^3}^2.
\end{align}
For the last term in \eqref{gaojie19}, there holds
\begin{align}\label{gaojie22}
\int_{\T^2}\la^{s}m\la^{s}\u\cdot\nabla m\,dx
\le& C\norm{\la^{s}m}{L^2}\norm{\la^{s}\u}{L^2}\norm{\nabla m}{L^\infty}\nn\\
\le&C\norm{\u}{H^3}\norm{m}{H^3}^2.
\end{align}
Hence, combining with \eqref{gaojie18}--\eqref{gaojie22}, we can get
\begin{align}\label{gaojie23}
\left|\int_{\T^2}\la^{s} (\u\cdot\nabla m+m\,\div \u)\cdot\la^{s} m\,dx
+\int_{\T^2}\la^{s} (m\nabla m)\cdot\la^{s} \u\,dx\right|
\le&C\norm{\u}{H^3}\norm{m}{H^3}^2.
\end{align}
Next, we have to bound the most trouble term
 \begin{align}\label{gaojie24}
\int_{\T^2}\la^{s}(a\,\div \u)\cdot\la^s a\,dx
 =\int_{\T^2}[\la^{s}, \, a\,\div] \u\cdot\la^s a\,dx-\int_{\T^2}a\,\div \la^{s}\u\cdot\la^s a\,dx.
 \end{align}
 
By Lemma \ref{jiaohuanzi}, we have
\begin{align}\label{gaojie25}
\left|-\int_{\T^2}[\la^{s},\, a\,\div] \u\cdot\la^s a\,dx\right|
\le C\norm{\u}{H^3}\norm{a}{H^3}^2.
\end{align}
Then, we have to deal with the last term on the right hand side of \eqref{gaojie24}. In fact, for $s=0,1,2$, we can bound this term directly as follows
\begin{align}\label{gaojie26}
\left|-\int_{\T^2}a\,\div\la^{s} \u\cdot\la^s a\,dx\right|
\le& C\norm{a}{L^\infty}\norm{\la^{s}\div \u}{L^2}\norm{\la^{s}a}{L^2}\nn\\
\le& C\norm{a}{H^2}\norm{\div\u}{H^2}\norm{a}{H^2}\nn\\
\le& C\norm{\u}{H^3}\norm{a}{H^3}^2.
\end{align}
However, if $s=3$,  the term $\norm{\div\u}{H^3}$  must  appear which   will lose control.
To overcome the difficulty, we  deduce  from the first equation of \eqref{mm3} that
 $$\div\u=-\frac{\partial_t a+\u\cdot\nabla a}{1+a},$$
 from which we have
\begin{align}\label{gaojie27}
-\int_{\T^2}a\la^{{3}} \div\u\cdot\la^{{3}} a\,dx
=&\int_{\T^2}a\la^{{3}} \big(\frac{\partial_t a+\u\cdot\nabla a}{1+a}\big)\cdot\la^{{3}} a\,dx \nn\\
=&\int_{\T^2}a\la^{{3}} \big(\frac{\partial_t a}{1+a}\big)\cdot\la^{{3}} a\,dx +\int_{\T^2}a\la^{{3}} \big(\frac{\u\cdot\nabla a}{1+a}\big)\cdot\la^{{3}} a\,dx\nn\\
\stackrel{\mathrm{def}}{=}&D_1+D_2.
\end{align}
For the first term $D_1$,  there holds
\begin{align}\label{gaojie28}
D_1=&\int_{\T^2}a\la^{{3}} \big(\frac{\partial_t a}{1+a}\big)\cdot\la^{{3}} a\,dx\nn\\
=&
\int_{\T^2}\frac{a}{1+a}\la^{{3}} ({\partial_t a})\cdot\la^{{3}} a\,dx+\int_{\T^2}a\sum_{\ell=0}^2 C_3^\ell\la^{{\ell}}\partial_t a
 \la^{{3-\ell}}\big(\frac{1}{1+a}\big)
 \cdot\la^{{3}} a\,dx
\nn\\
=&\frac{1}{2}\int_{\T^2}\frac{a}{1+a}\partial_t(\la^{{3}}a)^2 \,dx+\int_{\T^2}a\sum_{\ell=0}^2 C_3^\ell\la^{{\ell}}\partial_t a
 \la^{{3-\ell}}\big(\frac{1}{1+a}\big)
\cdot\la^{{3}} a\,dx
\nn\\
=&\frac{1}{2}\frac{d}{dt}\int_{\T^2}\frac{a}{1+a}(\la^{{3}}a)^2 \,dx-\frac{1}{2}\int_{\T^2}\frac{1}{(1+a)^2}\partial_t a(\la^{{3}}a)^2\,dx\nn\\
&+
\int_{\T^2}a\sum_{\ell=0}^2 C_3^\ell\la^{{\ell}}\partial_t a
\la^{{3-\ell}}\big(\frac{1}{1+a}\big)
 \cdot\la^{{3}} a\,dx.
\end{align}
Using the first equation of \eqref{mm3}, we  can bound the second term on the right hand side of \eqref{gaojie28} as
\begin{align}\label{gaojie29}
-\frac{1}{2}\int_{\T^2}\frac{1}{(1+a)^2}\partial_t a(\la^{{3}}a)^2\,dx
=&\frac{1}{2}\int_{\T^2}\frac{1}{(1+a)^2}(\u\cdot\nabla a+a\,\div \u+\div \u)(\la^{{3}}a)^2\,dx\nn\\
\le&C\big((1+\norm{a}{L^\infty})\norm{\nabla \u}{L^\infty}
       +\norm{\nabla a}{L^\infty}\norm{ \u}{L^\infty}\big)\norm{\la^{{3}}a}{L^2}^2\nn\\
\le&C\big(1+\norm{a}{H^3}\big)\norm{\u}{H^3}\norm{a}{H^3}^2.
\end{align}
By the H\"older inequality, the last term in \eqref{gaojie28} can be controlled as
\begin{align*}
&\int_{\T^d}a\sum_{\ell=0}^2 C_3^\ell\nabla^{{\ell}}\partial_ta
 \nabla^{{3-\ell}}\big(\frac{1}{1+a}\big)
 \cdot\nabla^{{3}} a\,dx\\
& \quad\leq C\norm{a}{L^\infty}\norm{\nabla^{3}a}{L^2}
 \Big(\norm{\partial_ta}{L^\infty}\Big\|{\nabla^{{3}}\big(\frac{1}{1+a}\big)}\Big\|_{L^2}\\
 &\qquad
  +\norm{\nabla\partial_ta}{L^3}\Big\|{\nabla^{{2}}\big(\frac{1}{1+a}\big)}\Big\|_{L^6}
   +\norm{\nabla^2\partial_ta}{L^2}\Big\|{\nabla\big(\frac{1}{1+a}\big)}\Big\|_{L^\infty}
  \Big)\\
&\quad\leq C\norm{a}{L^\infty}\norm{\nabla^{3}a}{L^2}\norm{\partial_ta}{H^2}
   \big(\norm{\nabla^{3}a}{L^2}+\norm{\nabla^{2}a}{H^1}+\norm{\nabla a}{H^2}\big)\\
&\quad\leq C\norm{\partial_ta}{H^2}\norm{ a}{H^3}^3.
\end{align*}
Using the fact that $H^2(\T^d)$ is Banach algebra, we get
\begin{align*}
\norm{\partial_t a}{H^2}
\leq C\norm{\u\cdot\nabla a+  a\,\div \u+\div \u}{H^2}
\leq C(1 +\norm{a}{H^3})\norm{\u}{H^3}.
\end{align*}
Hence, we have
\begin{align}\label{gaojie33}
\int_{\T^2}a\sum_{\ell=0}^2 C_3^\ell\la^{{\ell}}\partial_t a
 \la^{{3-\ell}}\big(\frac{1}{1+a}\big)
 \cdot\la^{{3}} a\,dx
 \le C(1+\norm{a}{H^3})\norm{\u}{H^3}\norm{a}{H^3}^3.
\end{align}
Combining  \eqref{gaojie29} with \eqref{gaojie33} gives
\begin{align}\label{gaojie34}
D_1
\le&\frac{1}{2}\frac{d}{dt}\int_{\T^2}\frac{a}{1+a}(\la^{{3}}a)^2 \,dx
+C(1+\norm{a}{H^3}^2)\norm{\u}{H^3}\norm{a}{H^3}^2.
\end{align}
For the term $D_2$, we get that
\begin{align}\label{gaojie35}
D_2=&\int_{\T^2}a\la^{{3}} \big(\frac{\u\cdot\nabla a}{1+a}\big)\cdot\la^{{3}} a\,dx\nn\\
=&\int_{\T^2}\frac{a}{1+a}\la^{{3}}(\u\cdot\nabla a)\cdot\la^{{3}} a\,dx
   +\int_{\T^2}a\sum_{\ell=0}^2 C_3^\ell\la^{{\ell}}
   (\u\cdot\nabla a)\la^{{3-\ell}}\big(\frac{1}{1+a}\big)
    \cdot\la^{{3}} a\,dx\nn\\
=&D_{2,1}+D_{2,2}.
\end{align}
We can use the commutator   to rewrite $D_{2,1}$ into
\begin{align*}
D_{2,1}=&\int_{\T^2}\frac{a}{1+a}
        \left[\la^{{3}},\,\u\cdot\nabla\right] a\cdot \la^{{3}} a\,dx+
\int_{\T^2}\frac{a}{1+a}\u\cdot\nabla\la^{{3}} a\cdot\la^{{3}} a\,dx.
\end{align*}
Thanks to Lemma \ref{jiaohuanzi}, we get
\begin{align*}
\left|\int_{\T^2}\frac{a}{1+a}
      \left[\la^{{3}},\,\u\cdot\nabla\right] a\cdot \la^{{3}} a\,dx\right|
\le& C\norm{\frac{a}{1+a}}{L^\infty}\norm{[\la^3,\u\cdot\nabla ]a}{L^2}\norm{\la^{{3}}a}{L^2}\nn\\
\le&C\big(\norm{\nabla \u}{L^\infty}\norm{\la^{{3}}a}{L^2}
    +\norm{\la^{{3}} \u}{L^2}\norm{\nabla a}{L^\infty}\big)\norm{\la^{{3}}a}{L^2}\nn\\
\le&C\norm{\u}{H^3}\norm{a}{H^3}^2.
\end{align*}
By using the integration by parts, we have
\begin{align*}
\left|
\int_{\T^2}\frac{a}{1+a}\u\cdot\nabla\la^{{3}} a\cdot\la^{{3}} a\,dx\right|\le& C\norm{\div\big(\frac{a\u}{1+a}\big)}{L^\infty}\norm{\la^{{3}}a}{L^2}^2\nn\\
\le&C\big(\norm{\nabla \u}{L^\infty}
     +\norm{ \u}{L^\infty}\norm{\nabla a}{L^\infty}\big)\norm{\la^{{3}}a}{L^2}^2\nn\\
\le&C\big(1+\norm{a}{H^3}\big)\norm{\u}{H^3}\norm{a}{H^3}^2,
\end{align*}
from which we get
\begin{align}\label{gaojie36}
D_{2,1}\le C\big(1+\norm{a}{H^3}\big)\norm{\u}{H^3}\norm{a}{H^3}^2.
\end{align}
Thanks to the H\"older inequality, we can get
\begin{align*}
D_{2,2}=&
\int_{\T^2}a\sum_{\ell=0}^2 C_3^\ell\la^{{\ell}}(\u\cdot\nabla a)\la^{{3-\ell}}\big(\frac{1}{1+a}\big)
\cdot\la^{{3}} a\,dx\nn\\
\leq& C\norm{a}{L^\infty}\norm{\nabla^{3}a}{L^2}
    \Big(\norm{\u\cdot\nabla a}{L^\infty}\Big\|{\nabla^{{3}}
      \big(\frac{1}{a+1}\big)}\Big\|_{L^2}\\
   &\qquad
  +\norm{\nabla(\u\cdot\nabla a)}{L^3}
     \Big\|{\nabla^{{2}}\big(\frac{1}{1+a}\big)}\Big\|_{L^6}
   +\norm{\nabla^2(\u\cdot\nabla a)}{L^2}\Big\|{\nabla\big(\frac{1}{1+a}\big)}\Big\|_{L^\infty}
  \Big)\\
\leq&C\norm{a}{L^\infty}\norm{\nabla^{3}a}{L^2}\norm{\u\cdot\nabla a}{H^2}
   (\norm{\nabla^{3}a}{L^2}+\norm{\nabla^{2}a}{H^1}+\norm{\nabla a}{H^2})\\
\leq&C\norm{\u}{H^3}\norm{ a}{H^3}^4,
\end{align*}
which combines with \eqref{gaojie36} implies that
\begin{align}\label{gaojie37}
D_{2}
\le C(1+\norm{a}{H^3}^2)\norm{\u}{H^3}\norm{a}{H^3}^2.
\end{align}
Inserting \eqref{gaojie34} and \eqref{gaojie37} into \eqref{gaojie27} leads to
\begin{align}\label{gaojie38}
-\int_{\T^2}a\la^{{3}} \div\u\cdot\la^{{3}} a\,dx
\le&\frac{1}{2}\frac{d}{dt}\int_{\T^2}\frac{a}{1+a}(\la^{{3}}a)^2 \,dx+C(1+\norm{a}{H^3}^2)\norm{\u}{H^3}\norm{a}{H^3}^2.
\end{align}

Consequently, taking the estimates \eqref{gaojie25}, \eqref{gaojie26}, and \eqref{gaojie38} into \eqref{gaojie24}, we get
\begin{align}\label{gaojie39}
-\int_{\T^2}\la^{s} (a\div\u)\cdot\la^s a\,dx
\le&\frac{1}{2}\frac{d}{dt}\int_{\T^2}\frac{a}{1+a}(\la^{3}a)^2 \,dx
+C(1+\norm{a}{H^3}^2)\norm{\u}{H^3}\norm{a}{H^3}^2.
\end{align}
 In the following, we bound the terms of $f_1$ in \eqref{gaojie13}. To do this, we  write
\begin{align}\label{gaojie40}
\int_{\T^2}\la^{s} f_1\cdot\la^{s} \u\,dx=A_1+A_2+A_3+A_4+A_5
\end{align}
with
\begin{align*}
&A_1\stackrel{\mathrm{def}}{=}-\int_{\T^2}\la^{s} \big(\u\cdot\nabla \u\big)\cdot\la^{s} \u\,dx,\nn\\
&A_2\stackrel{\mathrm{def}}{=}\int_{\T^2}\la^{s} \big(I(a)\nabla \aa\big)\cdot\la^{s} \u\,dx,\nn\\
&
A_3\stackrel{\mathrm{def}}{=}\int_{\T^2}\la^{s} \big(\mu\nabla I(a)\nabla\u\big)\cdot\la^{s} \u\,dx,\nn\\
&A_4\stackrel{\mathrm{def}}{=}\int_{\T^2}\la^{s}\big( (\lambda+\mu)
      \nabla I(a)\div\u\big)\cdot\la^{s} \u\,dx,\nn\\
&A_5\stackrel{\mathrm{def}}{=}-\int_{\T^2}\la^{s} \big(\ta\nabla J(a)\big)\cdot\la^{s} \u\,dx.
\end{align*}
The term $A_1$ can be bounded as in \eqref{gaojie17} to get
\begin{align*}
\big|A_1\big|\le C\norm{\u}{H^3}^3.
\end{align*}
By Lemmas \ref{daishu} and (\ref{eq:smalla}), we have
\begin{align*}
\big|A_2\big|
\leq&C\Big(\|\nabla \aa\|_{L^{\infty}}\|I(a)\|_{H^{s-1}}
        +\|\nabla \aa\|_{H^{s-1}}\|I(a)\|_{L^{\infty}}\Big)\norm{\la^{s+1}\u}{L^2}\nonumber\\
\leq&\frac{\mu c_0}{64}\norm{\la^{s+1}\u}{L^2}^2
         +C\norm{a}{H^{s-1}}^2\norm{\aa}{H^3}^2
         +C\|I(a)\|_{H^{2}}^2\norm{\aa}{H^s}^2\nn\\
\leq&\frac{\mu c_0}{64}\norm{\la^{s+1}\u}{L^2}^2
         +C\norm{a}{H^{3}}^2\norm{\aa}{H^3}^2.
\end{align*}
Similarly,
\begin{align*}
\big|A_3\big|+\big|A_4\big|
\leq&C\big(\|\nabla I(a)\|_{L^{\infty}}\norm{\la^{s}\u}{L^2}
        +\|\nabla I(a)\|_{H^{s-1}}\|\nabla \u\|_{L^{\infty}}\big)\norm{\la^{s+1}\u}{L^2}\nonumber\\
\leq&\frac{\mu{c_0}}{64}\norm{\la^{s+1}\u}{L^2}^2
   +C\norm{a}{H^3}^2\norm{ \u}{H^3}^2
\end{align*}
and
\begin{align*}
\big|A_5\big|
\leq&C\big(\|\nabla J(a)\|_{L^{\infty}}\|\ta\|_{H^{s-1}}
        +\|\nabla J(a)\|_{H^{s-1}}\|\ta\|_{L^{\infty}}\big)\norm{\la^{s+1}\u}{L^2}\nonumber\\
         \leq&\frac{\mu{c_0}}{64}\norm{\la^{s+1}\u}{L^2}^2+C\norm{a}{H^3}^2
         \norm{\ta}{H^3}^2.
\end{align*}
Inserting the bounds for $A_1$ through $A_5$ into \eqref{gaojie40}, we get
\begin{align}\label{gaojie41}
\left|\int_{\T^2}\la^{s} f_1\cdot\la^{s} \u\,dx\right|
\leq\frac{3\mu{c_0}}{64}\norm{\la^{s+1}\u}{L^2}^2
+C\norm{\u}{H^3}^3
+C\norm{(\aa,\u,\ta)}{H^3}^2\norm{a}{H^3}^2.
\end{align}
Finally, we have to bound the terms in $f_2$. we first rewrite
\begin{align*}
\int_{\T^2}\la^{s} f_2\cdot\la^{s} \ta\,dx=A_{11}+A_{12}+A_{13}
\end{align*}
with
\begin{align*}
&A_{11}=-\int_{\T^2}\la^{s} \div (\ta\u)\cdot\la^{s} \ta\,dx,\\
&A_{12}=\kappa\int_{\T^2}\la^{s} (\nabla I(a)\nabla\ta)\cdot\la^{s} \ta\,dx,\\
&A_{13}=\int_{\T^2}\la^{s} \big(\frac{1}{1+a}(2\mu |D(\u)|^2+\lambda(\div \u)^2)\big)
     \cdot\la^{s} \ta\,dx.
\end{align*}
Due to Lemma \ref{daishu}, we obtain
\begin{align*}
\int_{\T^2}\la^{s} (\ta\div\u )\cdot\la^s \ta\,dx
\leq&C\norm{\ta\div\u}{H^s}\norm{\ta}{H^s}\nn\\
\leq&C\left(\norm{\ta}{L^\infty}\norm{\div\u}{H^s}
         +\norm{\div\u}{L^\infty}\norm{\ta}{H^s}\right)\norm{\ta}{H^s}\nn\\
\leq&C\norm{\la^{s+1}\u}{L^2}\norm{\ta}{H^3}^2+\norm{\u}{H^3}\norm{\ta}{H^3}^2\nn\\
\le&\frac{\mu c_0}{16}\norm{\la^{s+1}\u}{L^2}^2+C\big(\norm{\u}{H^3}+\norm{\ta}{H^3}^2\big)\norm{\ta}{H^3}^2,
\end{align*}
where we have used the inequality
\begin{align*}
\|\div\u\|_{H^{s}}\le C\|\la^{s+1}\u\|_{L^{2}}.
\end{align*}
By a similar derivation of \eqref{gaojie17}, there holds
\begin{align*}
\int_{\T^2}\la^{s} (\u\cdot\nabla\ta)\cdot\la^s \ta\,dx
\le C\norm{\u}{H^3}\norm{\ta}{H^3}^2.
\end{align*}
Then we get the bound of $A_{11}$
\begin{align*}
|A_{11}|\leq&\left|\int_{\T^2}\la^{s} (\ta\div\u )\cdot\la^s \ta\,dx\right|
    +\left|\int_{\T^2}\la^{s} (\u\cdot\nabla\ta)\cdot\la^s \ta\,dx\right|\nn\\
\le&\frac{\mu c_0}{16}\norm{\la^{s+1}\u}{L^2}^2
       +C\big(\norm{\u}{H^3}+\norm{\ta}{H^3}^2\big)\norm{\ta}{H^3}^2.
\end{align*}
It follows from H\"{o}lder inequality, Lemma \ref{daishu} and Young's inequality that
\begin{align*}
|A_{12}|=&\left|\kappa\int_{\T^2}\la^{s} (\nabla I(a)\nabla\ta )\cdot\la^s \ta\,dx\right|\nn\\
\leq&C\big(\|\nabla I(a)\|_{L^{\infty}}\norm{\la^{s}\ta}{L^2}
        +\|\nabla I(a)\|_{H^{s-1}}\|\nabla \ta\|_{L^{\infty}}\big)\norm{\la^{s+1}\ta}{L^2}\nonumber\\
\leq&\frac{\kappa c_0}{32}\norm{\la^{s+1}\ta}{L^2}^2+C\big(\|\nabla a\|_{L^{\infty}}^2\norm{\la^{s}\ta}{L^2}^2
        +\norm{a}{H^s}^2\norm{ \ta}{H^3}^2\big)\nn\\
\leq& \frac{\kappa c_0}{32}\norm{\la^{s+1}\ta}{L^2}^2  +C\norm{a}{H^3}^2
         \norm{\ta}{H^3}^2.
\end{align*}
For the bound of $A_{13}$, it follows from the fact $H^2(\T^2)$ is Banach algebra that
\begin{align*}
|A_{13}|=&\left|\int_{\T^2}\la^{s} \big(\frac{1}{1+a}(2\mu |D(\u)|^2
          +\lambda(\div \u)^2)\big)\cdot\la^s \ta\,dx\right|\nonumber\\
\leq& C\norm{\frac{1}{1+a}(2\mu |D(\u)|^2
        +\lambda(\div \u)^2)}{H^2}
        \norm{\la^{s+1}\ta}{L^2}\nonumber\\
\leq& C\big(1+\norm{I(a)}{L^\infty}\big)
        \big(\norm{ |D(\u)|^2}{H^2}
        +\norm{(\div \u)^2}{H^2}\big)\norm{\la^{s+1}\ta}{L^2}\nn\\
\leq&\frac{\kappa c_0}{32}\norm{\la^{s+1}\ta}{L^2}^2+C(1+\norm{a}{H^3}^2)\norm{\u}{H^3}^4.
\end{align*}
Hence, we have
\begin{align}\label{gaojie43}
\left|\int_{\T^2}\la^{s} f_2\cdot\la^{s} \ta\,dx\right|\leq&|A_{11}|+|A_{12}|+|A_{13}|
\leq\frac{\mu c_0}{16}\norm{\la^{s+1}\u}{L^2}^2+\frac{\kappa c_0}{16}\norm{\la^{s+1}\ta}{L^2}^2\nn\\
     &+ C\big(\norm{\u}{H^3}+\norm{a}{H^3}^2+\norm{\ta}{H^3}^2\big)\norm{\ta}{H^3}^2
       +C\big(1+\norm{a}{H^3}^2\big)\norm{\u}{H^3}^4.
\end{align}
Finally, inserting \eqref{gaojie17}, \eqref{gaojie23}, \eqref{gaojie39},
   \eqref{gaojie41} and \eqref{gaojie43} into \eqref{gaojie13}
and then summing up  \eqref{gaojie13} over $1\leq s\leq3$, we obtain
\begin{align*}
&\frac12\frac{d}{dt}\norm{(a,\u,\ta,m)}{H^{3}}^2
     -\frac{1}{2}\frac{d}{dt}\int_{\T^2}\frac{a}{1+a}(\Lambda^3a)^2dx
      +\mu\norm{\nabla\u}{H^{3}}^2
     +(\lambda+\mu)\norm{\div\u}{H^{3}}^2+\kappa\norm{\nabla\ta}{H^{3}}^2\nn\\
& \quad\leq C\big((1+\norm{a}{H^3}^2)(1+\norm{\u}{H^3})\norm{\u}{H^3}
       +\norm{(\aa,\u,\ta)}{H^3}^2\big)
      \|( a, \u,\ta, m)\|_{H^{3}}^2.
\end{align*}
Consequently, we prove Lemma \ref{gaojie}.
\end{proof}

\subsection{Energy estimates for $(\aa,\u, \ta)$ }
In this subsection, we shall give the energy estimates for the unknown good function $\aa$,
We need to reformulate \eqref{mm3} in terms of variables $\aa$, $\u$ and $\ta$.
Precisely, one has
\begin{eqnarray}\label{energy1}
\left\{\begin{aligned}
&\partial_t \aa+ \div\u  =f_3,\\
&\partial_t \u-\div(\bar{\mu}(\rho)\nabla\u)-\nabla(\bar{\lambda}(\rho)\div\u)+\nabla \aa
+\nabla\ta=f_1,\\
&\partial_t\theta-\div(\bar{\kappa}(\rho)\nabla\ta)
+ \div \u=f_2,\\
&(\aa,\u,\ta)|_{t=0}=(\aa_0,\u_0,\ta_0),
\end{aligned}\right.
\end{eqnarray}
with
\begin{align*}
&\aa \stackrel{\mathrm{def}}{=}a+\frac{1}{2}m^2,\\
&f_3\stackrel{\mathrm{def}}{=}-\u\cdot\nabla \aa-\aa\div \u-\frac{1}{2}m^2\div\u,
\end{align*}
and $f_1$ is defined in \eqref{f1}, $f_2$ is defined in \eqref{f2}.
 Then, we will present the energy estimates for $(\aa,\u, \ta)$ in the following Lemma.
\begin{lemma}\label{haosan}
Let $(\aa,\u,\ta) \in C([0, T];H^3)$ be a solution to the  system \eqref{energy1}, then there holds
\begin{align}\label{energy2}
&\frac12\frac{d}{dt}\norm{(\aa,\u,\ta)}{H^{3}}^2+\mu\norm{\nabla\u}{H^{3}}^2
+(\lambda+\mu)\norm{\div\u}{H^{3}}^2+\kappa\norm{\nabla\ta}{H^{3}}^2\\
& \quad\leq C\big(\norm{(a,\u)}{H^3}
       +(1+\norm{a}{H^3}^2)\norm{\u}{H^3}^2
       +\norm{m}{H^3}^4 +\norm{(a,\aa,\ta)}{H^3}^2\big)
      \|(\aa, \u,\ta)\|_{H^{3}}^2.\nn
\end{align}
\end{lemma}
\begin{proof}
We start with the $L^2$ estimate.
Taking inner product with $({\aa},\u,\ta)$ for the equation \eqref{energy1},  we obtain
\begin{align}\label{energy3}
&\frac12\frac{d}{dt}\norm{(\aa,\u,\ta)}{L^2}^2-\int_{\T^2}\div(\bar{\mu}(\rho)\nabla\u)\cdot \u\,dx-\int_{\T^2}\nabla(\bar{\lambda}(\rho)\div\u)\cdot \u\,dx-\int_{\T^2}\div(\bar{\kappa}(\rho)\nabla\ta)\cdot \ta\,dx\nn\\
&\quad= \int_{\T^2} f_3\cdot \aa\,dx+\int_{\T^2} f_1\cdot \u\,dx+\int_{\T^2} f_2\cdot \ta\,dx,
\end{align}
where we have used the following cancellations
\begin{align*}
\int_{\T^2}\div\u\cdot \aa\,dx+\int_{\T^2}\nabla \aa\cdot\u\,dx
=\int_{\T^2}\div\u\cdot \ta\,dx+\int_{\T^2}\nabla \ta\cdot\u\,dx=0.
\end{align*}
For the last three terms on the left hand side of \eqref{energy3}, we get by integration by parts  and \eqref{youjiexing}  that
\begin{align}
-\int_{\T^2}\div(\bar{\mu}(\rho)\nabla\u)\cdot \u\,dx
=&\int_{\T^2}\bar{\mu}(\rho)\nabla\u\cdot \nabla\u\,dx
\ge \frac{\mu c_0} {2}\norm{\nabla\u}{L^{2}}^2,\label{energy4}\\
-\int_{\T^2}\nabla(\bar{\lambda}(\rho)\div\u)\cdot \u\,dx
=&\int_{\T^2}\bar{\lambda}(\rho)\div\u\cdot \div\u\,dx
  \ge \frac{\left(\lambda+\mu\right) c_0} {2}\norm{\div\u}{L^{2}}^2,\label{energy5}
\end{align}
and
\begin{align}
-\int_{\T^2}\div(\bar{\kappa}(\rho)\nabla\ta)\cdot \ta\,dx
=&\int_{\T^2}\bar{\kappa}(\rho)\nabla\ta\cdot \nabla\ta\,dx
\ge\frac{\kappa c_0} {2}\norm{\nabla\ta}{L^{2}}^2.\label{energy5+1}
\end{align}
Next, we shall estimate each term on the right hand side of \eqref{energy3}.
First, it follows from integration by parts and the H\"older inequality that
\begin{align}\label{energy6}
\Big|\int_{\T^2} f_3\cdot \aa\,dx\Big|
\le&\left|\int_{\T^2} (\u\cdot \nabla\aa)\cdot\aa\,dx\right|
   +\left|\int_{\T^2} \left(\aa\div\u\right)\cdot\aa\,dx\right|
   +\frac{1}{2}\left|\int_{\T^2}m^2\div\u\cdot\aa\,dx\right|\nn\\
\le&C\norm{\div\u}{L^{\infty}}\norm{\aa}{L^{2}}^2
    +C\norm{m}{L^{\infty}}^2\norm{\div\u}{L^{2}}\norm{\aa}{L^{2}}\nn\\
\le&\frac{\left(\lambda+\mu\right) c_0} {16}\norm{\div\u}{L^{2}}^2
      + C (\norm{\u}{H^{3}}
      +\norm{m}{H^{3}}^4)\norm{\aa}{L^{2}}^2.
\end{align}
For the first term in $f_1$,
\begin{align}\label{energy7}
\Big| \int_{\T^2} \u\cdot\nabla \u\cdot \u\,dx\Big|\le C \norm{\nabla\u}{L^{\infty}}\norm{\u}{L^{2}}^2
\le C \norm{\u}{H^{3}}^3.
\end{align}
Thanks to \eqref{P} and Lemma \ref{fuhe}, we have
\begin{align}\label{energy9}
\Big| \int_{\T^2} I(a)\nabla \aa\cdot \u\,dx\Big|
\le& C \norm{I(a)}{L^{\infty}}\norm{ \nabla \aa}{L^{2}}\norm{\u}{L^{2}}\nn\\
\le& \frac{\mu c_0} {32}\norm{\u}{L^{2}}^2+C \norm{I(a)}{H^{2}}^2\norm{ \aa}{H^{1}}^2,\nn\\
\le& \frac{\mu c_0} {32}\norm{\nabla\u}{L^{2}}^2+C \norm{a}{H^{3}}^2\norm{ \aa}{H^{3}}^2,
\end{align}
\begin{align}
\label{energy10}
&\mu\Big| \int_{\T^2} \nabla I(a)\nabla\u\cdot \u\,dx\Big|
     +(\lambda+\mu)\Big|\int_{\T^2} \nabla I(a)\div\u\cdot \u\,dx\Big|\nn\\
&\quad\le C \norm{\nabla I(a)}{L^{\infty}}\norm{\nabla \u}{L^{2}}\norm{\u}{L^{2}}\nn\\
&\quad\le\frac{\mu c_0} {32}\norm{\nabla\u}{L^{2}}^2+C\norm{\nabla I(a)}{H^{2}}^2\norm{\u}{L^{2}}^2\nn\\
&\quad\le\frac{\mu c_0} {32}\norm{\nabla\u}{L^{2}}^2+C\norm{a}{H^{3}}^2\norm{\u}{H^{3}}^2,
\end{align}
and
\begin{align}\label{energy11}
\Big| \int_{\T^2} \ta \nabla J(a)\cdot \aa\,dx\Big|
\le& C \norm{\nabla J(a)}{L^{\infty}}\norm{\ta}{L^{2}}\norm{\aa}{L^{2}}\nn\\
\le&C\norm{a}{H^{3}}\norm{\aa}{L^{2}}\norm{\ta}{L^{2}}\nn\\
\le&C\norm{a}{H^{3}}\big(\norm{\ta}{H^{3}}^2+\norm{\aa}{H^{3}}^2\big).
\end{align}
Combining with  \eqref{energy7}, \eqref{energy9}-\eqref{energy11} gives
\begin{align}\label{energy12}
\Big| \int_{\T^2} f_1\cdot \u\,dx\Big|
\le\frac{\mu c_0} {16}\norm{\nabla\u}{L^{2}}^2
   +C\big(\norm{(a,\u)}{H^{3}}+  \norm{a }{H^{3}}^2\big)
  \norm{(\aa,\u,\ta)}{H^3}^2.
\end{align}
In the following, we bound the terms in $f_2$. To do so, we write
\begin{align}\label{energy13}
\Big| \int_{\T^2} f_2\cdot \ta\,dx\Big|
=\sum_{i=1}^{3}A_{2i}
\end{align}
with
\begin{align*}
&A_{21}\stackrel{\mathrm{def}}{=}-\int_{\T^2}\div(\ta\u)\cdot\ta\,dx,\qquad\qquad
A_{22}\stackrel{\mathrm{def}}{=}\kappa\int_{\T^2}(\nabla I(a))\nabla\ta\cdot \ta\,dx,\\
&A_{23}\stackrel{\mathrm{def}}{=}
      \int_{\T^2} \big(\frac{1}{1+a}(2\mu |D(\u)|^2+\lambda(\div \u)^2)\big)\cdot \ta\,dx.
\end{align*}
The term $A_{21}$ can be bounded as
\begin{align*}
\left|A_{21}\right|\leq&
\Big| \int_{\T^2} \u\cdot\nabla \ta\cdot \ta\,dx\Big|
    +\Big| \int_{\T^2} \ta\div\u\cdot \ta\,dx\Big|\nn\\
\le& C \norm{\u}{L^{\infty}}\norm{\nabla\ta}{L^{2}}\norm{\ta}{L^{2}}
       +C \norm{\ta}{L^{\infty}}\norm{\div\u}{L^{2}}\norm{\ta}{L^{2}}\nn\\
\le&\frac{\kappa c_0}{32}\norm{\nabla\ta}{L^{2}}^2
       +\frac{\left(\lambda+\mu\right) c_0}{16}\norm{\div\u}{L^{2}}^2
       +C \big(\norm{\u}{H^{3}}^{2}+\norm{\ta}{H^{3}}^{2}\big)\norm{\ta}{L^{2}}^{2}.
\end{align*}
For term $A_{22}$, by Lemma \ref{fuhe} and the Young's inequality, we have
\begin{align*}
\left|A_{22}\right|
\le C \norm{\nabla I(a)}{L^{\infty}}\norm{\nabla \ta}{L^{2}}\norm{\ta}{L^{2}}
\le \frac{\kappa c_0}{32}\norm{\nabla\ta}{L^{2}}^2
      +C\norm{ a}{H^{3}}^2\norm{\ta}{H^{3}}^2.
\end{align*}
Similarly, for the term  $A_{23}$, there holds
\begin{align*}
\left|A_{23}\right|
=&\Big| \int_{\T^2} \big(\frac{1}{1+a}(2\mu |D(\u)|^2+\lambda(\div \u)^2)\big)\cdot \ta\,dx\Big|\nn\\
\le& C (1+\norm{ I(a)}{L^{\infty}})\norm{ \nabla\u}{L^{\infty}}\norm{\nabla \u}{L^{2}}\norm{\ta}{L^{2}}\nn\\
\le& \frac{\mu c_0}{16}\norm{\nabla\u}{L^{2}}^2
        +C(1+\norm{ a}{H^{3}}^2)\norm{\u}{H^{3}}^{2}\norm{\ta}{H^{3}}^{2}.
\end{align*}
Plugging $A_{21}$-$A_{23}$ into \eqref{energy13}, we get
\begin{align}\label{energy14}
\Big| \int_{\T^2} f_2\cdot \ta\,dx\Big|
\le&\frac{\mu c_0}{16}\norm{\nabla\u}{L^{2}}^2
     +\frac{\left(\lambda+\mu\right) c_0}{16}\norm{\div\u}{L^{2}}^2
     +\frac{\kappa c_0}{16}\norm{\nabla\ta}{L^{2}}^2 \nn\\
  &+C\big((1+\norm{a}{H^3}^2)\norm{\u}{H^3}^2+\norm{(a,\ta)}{H^3}^2\big)\norm{\ta}{H^3}^2.
\end{align}
Inserting  \eqref{energy6}, \eqref{energy12} and \eqref{energy14} into \eqref{energy3}
 and using \eqref{energy4}-\eqref{energy5+1}, we arrive at a basic energy inequality
\begin{align}\label{energy15}
&\frac12\frac{d}{dt}\norm{\left(\aa,\u,\ta\right)}{L^2}^2+\mu\norm{\nabla\u}{L^2}^2+(\lambda+\mu)
\norm{\div\u}{L^2}^2+\kappa\norm{\nabla\ta}{L^2}^2\nn\\
&\quad\le C\big(\norm{(a,\u)}{H^{3}}
   +\norm{(a,\ta)}{H^{3}}^2+\norm{m}{H^{3}}^4+(1+\norm{a}{H^{3}}^2)\norm{\u}{H^{3}}^2
    \big)
    \norm{(\aa,\u,\ta)}{H^3}^2.
\end{align}

 Next, we are concerned with the higher energy estimates.
 Applying  $\la^s$ with $1\le s\le 3$ to \eqref{energy1}
 and then taking $L^2$ inner product with $(\la^s\aa, \la^s\u,\la^s\ta)$ yields
\begin{align}\label{energy16}
&\frac12\frac{d}{dt}\norm{\left(\la^s\aa,\la^s\u,\la^s\ta\right)}{L^2}^2
-\int_{\T^2}\la^{s}\div(\bar{\mu}(\rho)\nabla\u)\cdot\la^{s} \u\,dx\nn\\
&\qquad-\int_{\T^2}\la^{s}\nabla(\bar{\lambda}(\rho)\div\u)\cdot\la^{s} \u\,dx-\int_{\T^2}\la^{s}\div(\bar{\kappa}(\rho)\nabla\ta)\cdot\la^{s} \ta\,dx\nn\\
&\quad=\int_{\T^2}\la^{s} f_3\cdot\la^{s} \aa\,dx+\int_{\T^2}\la^{s} f_1\cdot\la^{s} \u\,dx+\int_{\T^2}\la^{s} f_2\cdot\la^{s} \ta\,dx.
\end{align}
The last three terms on the left hand side of \eqref{energy16}
can be dealt from \eqref{gaojie10}-\eqref{gaojie12} that
\begin{align}\label{energy17}
&\frac12\frac{d}{dt}\norm{\big(\la^s\aa,\la^s\u,\la^s\ta\big)}{L^2}^2
   +\frac{\mu c_0}{4}\norm{\la^{s+1}\u}{L^2}^2
   +\frac{(\lambda+\mu)c_0}{4}\norm{\la^{s}\div\u}{L^2}^2
   +\frac{\kappa c_0}{4}\norm{\la^{s+1}\ta}{L^2}^2 \nn\\
&\quad\le C\int_{\T^2}\la^{s} f_3\cdot\la^{s} \aa\,dx
+C\int_{\T^2}\la^{s} f_1\cdot\la^{s} \u\,dx+C\int_{\T^2}\la^{s} f_2\cdot\la^{s} \ta\,dx
\nn\\
&\qquad  + C\big(\norm{ \u}{H^3}^2+\norm{ \ta}{H^3}^2\big)\norm{a}{H^3}^2.
\end{align}
We now estimate successively terms on the right hand side of \eqref{energy17}.
To bound the first term in $f_3$, we rewrite it into
\begin{align*}
\int_{\T^2}\la^{s} (\u\cdot\nabla \aa)\cdot\la^{s} \aa\,dx
=&\int_{\T^2}[\la^{s},\, \u\cdot\nabla] \aa\cdot\la^{s} \aa\,dx
+\int_{\T^2}\u\cdot\nabla\la^{s}\aa\cdot\la^{s} \aa\,dx.
\end{align*}
Then according to  Lemma \ref{jiaohuanzi} and integration by parts, we have
\begin{align}\label{energy18}
\Big|\int_{\T^2}\la^{s} (\u\cdot\nabla \aa)\cdot\la^{s} \aa\,dx\Big|
\le& C\norm{[\la^s,\,\u\cdot\nabla]\aa}{L^2}\norm{\la^{s}\aa}{L^2}
      +C\norm{\div\u}{L^\infty}\norm{\la^{s}\aa}{L^2}^2\nn\\
\le&C\big(\norm{\nabla \u}{L^\infty}\norm{\la^{s}\aa}{L^2}
      +\norm{\la^{s} \u}{L^2}\norm{\nabla \aa}{L^\infty}\big)\norm{\la^{s}\aa}{L^2}\nn\\
\le&C\norm{ \u}{H^3}\norm{\aa}{H^3}^2.
\end{align}
For the second term in $f_3$, it follows from Lemma \ref{daishu} that
\begin{align}\label{energy19}
\int_{\T^2}\la^{s} (\aa\div\u)\cdot\la^{s} \aa\,dx\leq&C\big(\|\div\u\|_{L^{\infty}}\|\aa\|_{H^{s}}
        +\|\div\u\|_{H^{s}}\|\aa\|_{L^{\infty}}\big)\norm{\la^{s}\aa}{L^2}\nonumber\\
         \leq&\frac{\mu c_0}{64}\|\nabla\u\|_{H^3}^2+C\|\aa\|_{H^{3}}^4.
\end{align}
Similarly, we have
\begin{align}\label{energy21}
\int_{\T^2}\la^{s} (m^2\div\u)\cdot\la^{s} \aa\,dx
\leq&C\big(\|\div\u\|_{L^{\infty}}\|m^2\|_{H^{s}}
        +\|\div\u\|_{H^{s}}\|m^2\|_{L^{\infty}}\big)\norm{\la^{s}\aa}{L^2}\nonumber\\
         \leq&\frac{\mu c_0}{64}\|\nabla\u\|_{H^3}^2+C\norm{m}{H^3}^4\norm{\aa}{H^3}^2.
\end{align}
Collecting \eqref{energy18}-\eqref{energy21}, we can get
\begin{align}\label{energy22}
\left|\int_{\T^2}\la^{s} f_3\cdot\la^{s} \aa\,dx\right|
\le&\frac{\mu{c_0}}{32}\norm{\nabla\u}{H^3}^2
+C\big(\|\u\|_{H^3}+\|\aa\|_{H^3}^2+\norm{m}{H^3}^4\big)
     \norm{\aa}{H^3}^2.
\end{align}
We get by a  similar  derivation of \eqref{gaojie41} and \eqref{gaojie43} that
\begin{align}\label{energy23}
\left|\int_{\T^2}\la^{s} f_1\cdot\la^{s} \u\,dx\right|
\leq&\frac{3\mu{c_0}}{16}\norm{\la^{s+1}\u}{L^2}^2
   +C\norm{\u}{H^3}^3+C\norm{(\aa,\u,\ta)}{H^3}^2\norm{a}{H^3}^2
\end{align}
and
\begin{align}\label{energy24}
\left|\int_{\T^2}\la^{s} f_2\cdot\la^{s} \ta\,dx\right|
\leq&\frac{\mu c_0}{16}\norm{\la^{s+1}\u}{L^2}^2+\frac{\kappa c_0}{16}\norm{\la^{s+1}\ta}{L^2}^2\nn\\
     &+ C\big(\norm{\u}{H^3}+\norm{a}{H^3}^2+\norm{\ta}{H^3}^2\big)\norm{\ta}{H^3}^2
       +C\big(1+\norm{a}{H^3}^2\big)\norm{\u}{H^3}^4.
\end{align}
Plugging \eqref{energy22}-\eqref{energy24} into \eqref{energy17} and combining with \eqref{energy15}, we arrive at the desired estimate \eqref{energy2}. This completes the proof of Lemma \ref{haosan}.
\end{proof}

\subsection{Dissipation estimates for $(\aa,\u,\ta,\mathbf{G})$}
Next, we find the hidden dissipation of $\aa.$
We rewrite \eqref{energy1} into
 \begin{eqnarray}\label{dissp1}
    \left\{\begin{aligned}
    &\partial_t \aa+ \div\u  =f_3,\\
    &\partial_t\u - \mu\Delta \u -(\lambda+\mu)\nabla \div \u+\nabla\ta
            +\nabla\aa=\mathbf{M}(a,\u,\ta,\aa),\\
    &\partial_t\theta-\kappa\Delta\theta+ \div \u=\mathbf{N}(a,\u,\ta),\\
    &(\aa,\u,\ta)|_{t=0}=(\aa_0,\u_0,\ta_0),
    \end{aligned}\right.
    \end{eqnarray}
  where
   \begin{align}\label{dissp2}
    &\mathbf{M}(a,\u,\ta,\aa)\stackrel{\mathrm{def}}{=}- \u\cdot \nabla \u
    +I(a)\nabla\aa-I(a)\Big(\mu\Delta \u + (\lambda+\mu)\nabla \div \u\Big)-\ta\nabla J(a),\\
    &\mathbf{N}(a,\u,\ta)\stackrel{\mathrm{def}}{=}- \div (\ta\u)- \kappa I(a)\Delta \theta
    +\frac{1}{1+a}\Big(2\mu |D(\u)|^2+\lambda(\div \u)^2\Big),\label{dissp3}
    \end{align}
with $I(a)=\frac{a}{1+a}$ and $J(a)=\mathrm{ln}(1+a).$
Denote
\begin{align}\label{dissp4}
{\mathbf{G}}\stackrel{\mathrm{def}}{=} \mathbb{Q}\u-\frac{1}{\nu}\Delta^{-1}\nabla {\aa}
\end{align}
with $\mathbb{Q}=\nabla\Delta^{-1}\div$ and $\nu=\lambda+2\mu>0$.
Then, we find out that ${\mathbf{G}}$ satisfies
\begin{align}\label{dissp5}
\partial_t{\mathbf{G}}-\nu\Delta {\mathbf{G}}
   =\frac{1}{\nu}\q \u+\mathbb{Q}\mathbf{M}-\nabla\ta
   +\frac{1}{\nu}\Delta^{-1}\nabla f_3.
\end{align}
The goal of this subsection is to establish the dissipation estimates for $(\aa,\u,\ta,\mathbf{G})$.
\begin{lemma}\label{miduhaosan}
Let $(\aa,\u,\ta,\mathbf{G}) \in C([0, T];H^3)$ be a solution
  to the  system \eqref{dissp1} and  \eqref{dissp5},
then there holds
\begin{align}\label{dissp6}	
&\frac12\frac{d}{dt}\norm{(\aa,\u,\ta,\G)}{H^{{3}}}^2+\norm{\aa}{H^{{3}}}^2
   +\norm{(\nabla\u,\nabla\ta,\nabla\G)}{H^{{3}}}^2\nn\\
&\quad\le C\norm{(a,\u)}{H^{{3}}}^2\norm{(\nabla\u,\nabla\ta)}{H^{{3}}}^2
+C\big(\|(a,\u)\|_{H^3}+\norm{ (a,\aa,\u,\ta)}{H^3}^2\nn\\
&\qquad
+\norm{m}{H^3}^4+\norm{ a}{H^3}^2\norm{\u}{H^3}^2\big)
\norm{(\aa,\u,\ta)}{H^{3}}^2.
\end{align}
\end{lemma}
\begin{proof}
For $s=0,$ we can follow  the derivation of \eqref{energy15} to get the desired estimates, here we omit the details. For $s=1,2,3$,
applying ${\la^s} $ to  the above equation \eqref{dissp5},
 and  taking the $L^2$-inner product with ${\la^s}\G$ give
\begin{align*}
	&\frac12\frac{d}{dt}\norm{\la^{s} \G}{L^{2}}^2+
       \nu\norm{\la^{s+1} \G}{L^{2}}^2\nn\\
	&\quad\le C\left|\int_{\T^2}\la^s\q \u\cdot\la^s \G\,dx\right|
     +C\left|\int_{\T^2}\la^s\nabla\ta\cdot\la^s \G\,dx\right|\nn\\
&\qquad+C\left|\int_{\T^2}\la^{s-1}f_3\cdot\la^s \G\,dx\right|
   +C\left|\int_{\T^2}\la^s\mathbb{Q}\mathbf{M}\cdot\la^s \G\,dx\right|.
\end{align*}
With the aid of the Young inequality,  we obtain
\begin{align}\label{dissp7}
	&\frac12\frac{d}{dt}\norm{\la^{s} \G}{L^{2}}^2+\frac{\nu}{2}\norm{\la^{s+1} \G}{L^{2}}^2\nn\\
	&\quad\le C\big(\norm{\la^{s-1}\q\u}{L^{2}}^2+\norm{\la^{s}\ta}{L^{2}}^2+\norm{\la^{s-2}f_3}{L^{2}}^2
	+\norm{\la^{s-1}\mathbf{M}}{L^{2}}^2\big).
\end{align}
By using the fact
$$\div\q\u=\div\u,$$
we infer from the first equation in \eqref{dissp1} and \eqref{dissp4}
that $\aa$ satisfies a damped transport equation
\begin{align*}
\partial_t{\aa}+\frac{1}{\nu}{\aa}=-\div {\mathbf{G}}+f_3.
\end{align*}
For the above equation, we get by a similar derivation of \eqref{dissp7} that
\begin{align*}
\frac12\frac{d}{dt}\norm{\la^s {\aa}}{L^{2}}^2+\frac{1}{\nu}\norm{\la^s {\aa}}{L^{2}}^2
=-\int_{\T^2}\la^s\div {\mathbf{G}}\cdot\la^s {\aa}\,dx+\int_{\T^2}\la^sf_3\cdot\la^s {\aa}\,dx.
\end{align*}
The last term on the right hand side of the above equality can be bounded the same as \eqref{energy22}
\begin{align*}
\left|\int_{\T^2}\la^{s} f_3\cdot\la^{s} \aa\,dx\right|
\le&\varepsilon\|\nabla\u\|_{H^3}^2
+C(\|\u\|_{H^3}+\|\aa\|_{H^3}^2
  +\norm{m}{H^3}^4)\norm{\aa}{H^3}^2,
\end{align*}
from which, we have
\begin{align}\label{dissp10}
\frac12\frac{d}{dt}\norm{\la^s {\aa}}{L^{2}}^2
   +\frac{1}{2\nu}\norm{\la^s {\aa}}{L^{2}}^2
   \le&\varepsilon\|\nabla\u\|_{H^3}^2+C\norm{\nabla\G}{H^{{3}}}^2\nn\\
&
+ C(\|\u\|_{H^3}+\norm{\aa}{H^3}^2
      +\norm{m}{H^3}^4)\norm{\aa}{H^3}^2.
\end{align}
From the third equation in \eqref{dissp1}, we get
\begin{align}\label{dissp11}
	&\frac12\frac{d}{dt}\norm{\la^{s} \ta}{L^{2}}^2+\frac{\kappa}{2}\norm{\la^{s+1} \ta}{L^{2}}^2\le C\big(\norm{\la^{s}\u}{L^{2}}^2+\norm{\la^{s-1}\mathbf{N}}{L^{2}}^2\big).
\end{align}
From  \eqref{dissp7}, \eqref{dissp10} and \eqref{dissp11},
by multiplying  suitable large constant lead to
\begin{align}\label{dissp12}
	&\frac12\frac{d}{dt}\norm{({\aa},\ta,\G)}{H^{{3}}}^2+\norm{\aa}{H^{{3}}}^2
	+\norm{\nabla\ta}{H^{{3}}}^2+\norm{\nabla\G}{H^{{3}}}^2\nn\\
	&\quad\le C\varepsilon\|\nabla\u\|_{H^3}^2+C\|\u\|_{H^3}^4
   + C\norm{\u}{H^{{3}}}^2+C\norm{f_3}{H^{{1}}}^2
   +C\norm{(\mathbf{M},\mathbf{N})}{H^{{2}}}^2\nn\\
    &\qquad+ C(\|\u\|_{H^3}
        +\|\aa\|_{H^3}^2
            +\norm{m}{H^3}^4)\norm{\aa}{H^3}^2,
\end{align}
where we have used the fact from Lemma \ref{poincareineq},
\begin{align*}
\norm{\ta}{H^{3}}^2
\approx \|{\ta}\|_{L^{2}}+\|{\ta}\|_{\dot{H}^{3}}^2
 \le C\big(\|\nabla\ta\|_{L^2}^2+\|\nabla\u\|_{L^2}^4+\|\aa\|_{L^2}^2+\norm{\nabla\ta}{H^{3}}^2\big).
\end{align*}
Using the fact that $H^2(\T^2)$ is Banach algebra, the nonlinear terms in \eqref{dissp12} can be estimated as follows. Due to
\begin{align*}
\norm{\u\cdot\nabla\aa}{H^{{1}}}^2+\norm{{\aa} \div\u}{H^{{1}}}^2
\le C\norm{\u}{H^{{2}}}^2\norm{\nabla\aa}{H^{{2}}}^2+C\norm{\aa}{H^{{2}}}^2\norm{\div\u}{H^{{2}}}^2
\leq C\norm{\u}{H^{{3}}}^2\norm{\aa}{H^{{3}}}^2,
\end{align*}
and
\begin{align*}
\norm{m^2 \div\u}{H^{{1}}}^2
\le C\norm{m}{H^{{3}}}^4\norm{\u}{H^{{3}}}^2,
\end{align*}
we obtain
\begin{align}\label{dissp13}
\norm{f_3}{H^{{1}}}^2
\le C(\norm{\aa}{H^{{3}}}^2+\norm{m}{H^{{3}}}^4
   )\norm{\u}{H^{{3}}}^2.
\end{align}
From \eqref{dissp2} and \eqref{dissp3}, by using Lemma \ref{fuhe}, we get
\begin{align}\label{dissp14}
\norm{\mathbf{M}}{H^{{2}}}^2
\le& C\norm{\u}{H^{{2}}}^2\norm{\nabla\u}{H^{{2}}}^2
   +C\norm{a}{H^{{2}}}^2\norm{\nabla \aa}{H^{{2}}}^2
    +C\norm{a}{H^{{2}}}^2\norm{\nabla\u}{H^{{3}}}^2
     +\norm{a}{H^{{2}}}^2\norm{\ta}{H^{{3}}}^2\nn\\
\le& C\norm{(a,\u)}{H^{{3}}}^2\norm{\nabla\u}{H^{{3}}}^2
      +C\norm{(\aa,\ta)}{H^{{3}}}^2\norm{a}{H^{{3}}}^2
\end{align}
and
\begin{align}\label{dissp15}
\norm{\mathbf{N}}{H^{{2}}}^2
\le& C\norm{\u}{H^{{2}}}^2\norm{\nabla\ta}{H^{{2}}}^2
+C\norm{\ta}{H^{{2}}}^2\norm{\div \u}{H^{{2}}}^2\nn\\
&+C\norm{I(a)}{H^{{2}}}^2\norm{\nabla\ta}{H^{{3}}}^2
+C(1+\norm{I(a)}{H^{{2}}}^2)\norm{\nabla\u}{H^{{2}}}^4
\nn\\
\le& C\big((1+\norm{a}{H^{{3}}}^2)\norm{\u}{H^{{3}}}^2+\norm{\ta}{H^{{3}}}^2\big)
\norm{\u}{H^{{3}}}^2
+C\norm{(a,\u)}{H^{{3}}}^2\norm{\nabla\ta}{H^{{3}}}^2.
\end{align}
Inserting \eqref{dissp13}-\eqref{dissp15} into \eqref{dissp12} leads to
\begin{align}\label{dissp16}
&\frac12\frac{d}{dt}\norm{(\aa,\ta,\G)}{H^{{3}}}^2+\norm{\aa}{H^{{3}}}^2
	+\norm{(\nabla\ta,\nabla\G)}{H^{{3}}}^2\nn\\
&\quad\le C\varepsilon\|\nabla\u\|_{H^3}^2+C\norm{\u}{H^{{3}}}^2
       +C\norm{(a,\u)}{H^{{3}}}^2\norm{\nabla\ta}{H^{{3}}}^2
       +C\norm{\ta}{H^{{3}}}^2\norm{a}{H^{{3}}}^2\nn\\
&\qquad+C\norm{(a,\u)}{H^{{3}}}^2\norm{\nabla\u}{H^{{3}}}^2
       +C\big((1+\norm{a}{H^{{3}}}^2)\norm{\u}{H^{{3}}}^2+\norm{\ta}{H^{{3}}}^2\big)
         \norm{\u}{H^{{3}}}^2\nn\\
&\qquad+ C\big(\|(a,\aa,\u)\|_{H^3}^2+\norm{ \u}{H^3}
        +\norm{m}{H^3}^4\big)\norm{(\aa,\u)}{H^3}^2.
\end{align}
The previous estimate \eqref{energy2} gives
\begin{align}\label{dissp17}
&\frac12\frac{d}{dt}\norm{(\aa,\u,\ta)}{H^{3}}^2+\mu\norm{\nabla\u}{H^{3}}^2
+(\lambda+\mu)\norm{\div\u}{H^{3}}^2+\kappa\norm{\nabla\ta}{H^{3}}^2\nn\\
& \quad\leq C\big(\norm{(a,\u)}{H^3}
       +(1+\norm{a}{H^3}^2)\norm{\u}{H^3}^2
       +\norm{m}{H^3}^4\nn\\
 & \qquad \quad \quad  +\norm{(a,\aa,\ta}{H^3}^2\big)
      \|(\aa, \u,\ta)\|_{H^{3}}^2.
\end{align}

Thus,  multiplying the above inequality \eqref{dissp17} by a suitable large constant and
then adding to  \eqref{dissp16}, we can finally get  that
\begin{align*}	
&\frac12\frac{d}{dt}\norm{(\aa,\u,\ta,\G)}{H^{{3}}}^2+\norm{\aa}{H^{{3}}}^2
   +\norm{(\nabla\u,\nabla\ta,\nabla\G)}{H^{{3}}}^2\nn\\
&\quad\le C\norm{(a,\u)}{H^{{3}}}^2\norm{(\nabla\u,\nabla\ta)}{H^{{3}}}^2
+C
\big(\|(a,\u)\|_{H^3}+\norm{( a,\aa,\u,\ta)}{H^3}^2\nn\\
&\qquad
+\norm{m}{H^3}^4+\norm{ a}{H^3}^2\norm{\u}{H^3}^2\big)
\norm{\aa,\u,\ta}{H^{3}}^2.
\end{align*}
This completes the proof of Lemma \ref{miduhaosan}.
\end{proof}

\subsection{Bootstrap argument}
In this section, we prove Theorem \ref{dingli}.
For any $(a_0, \u_0,\ta_0, m_0)\in H^{{3}}({ \mathbb{T} }^2)$,  we can prove the local well-posedness of \eqref{mm3}
by using the  standard energy method.  To prove the global well-posedness, the aim is
 on the global bound of
$(a,\u,\ta,m)$ in $H^3(\mathbb T^2)$. We use the bootstrapping argument and start by making the
ansatz that
\begin{align}\label{boot1}
	\sup_{t\in[0,T]}\norm{(a,\mathbf{u},\ta,m)}{H^{3}}\le \delta,
\end{align}
for suitably chosen $0<\delta<1$. The main efforts are devoted to proving that, if the initial norm is taken to be sufficiently small, namely
\begin{align*}
\norm{(a_0,\mathbf{u}_0,\ta_0,m_0)}{H^{{3}}}\le\varepsilon,
\end{align*}
with sufficiently small $\varepsilon>0$, then
\begin{align*}
\sup_{t\in[0,T]}\norm{(a,\mathbf{u},\ta,m)}{H^{3}}\le \frac{\delta}{2}.
\end{align*}
Under the assumption of \eqref{boot1}, we infer from \eqref{dissp6} that
\begin{align}\label{boot2}	&\frac12\frac{d}{dt}\norm{(\aa,\u,\ta,\G)}{H^{{3}}}^2+\norm{\aa}{H^{{3}}}^2+\norm{(\nabla\u,\nabla\ta,\nabla\G)}{H^{{3}}}^2
	\nn\\
	&\quad\le C
\delta(\delta^3+\delta+1)\norm{(\aa,\u,\ta)}{H^{3}}^2+C\delta^2\norm{(\nabla\u,\nabla\ta)}{H^{{3}}}^2,
\end{align}
where we use the fact
\begin{align*}
\norm{\aa}{H^{{3}}}^2=&\norm{a+\frac{1}{2}m^2}{H^{{3}}}^2
 \le\norm{a}{H^{{3}}}^2+\frac{1}{2}\norm{m^2}{H^{{3}}}^2
 \le\norm{a}{H^{{3}}}^2+\frac{1}{2}\norm{m}{H^{{3}}}^4
 \le C\delta^2.
\end{align*}
Denote
\begin{align*}
	{\mathcal{E}(t)}=&\norm{(\aa,\u,\ta,\G)}{H^{{3}}}^2
\end{align*}
and
\begin{align*}
	{\mathcal{D}(t)}=&\norm{\aa}{H^{{3}}}^2+\norm{(\nabla\u,\nabla\ta,\nabla\G)}{H^{{3}}}^2.
\end{align*}
Thanks to Lemma \ref{poincareineq}, there holds
\begin{align}\label{tapoin}
\norm{\ta}{H^{3}}^2
	\approx& \|{\ta}\|_{L^{2}}+\|{\ta}\|_{\dot{H}^{3}}^2\nn\\ \le&C\big(\|\nabla\ta\|_{L^2}^2+\|\nabla\u\|_{L^2}^4+\|\aa\|_{L^2}^2+\norm{\nabla\ta}{H^{3}}^2\big)\nn\\
\le&C\norm{\nabla\ta}{H^{3}}^2+C\delta^2\norm{\nabla\u}{H^{3}}^2+C\norm{\aa}{H^{3}}^2.
\end{align}
From \eqref{P} and \eqref{jinq1}, choosing $\delta$ small enough in \eqref{boot2} implies that
\begin{align}\label{boot3}
	\frac{d}{dt}{\mathcal{E}(t)}+\frac12{\mathcal{D}(t)}\le 0.
\end{align}
On the one hand, notice that $\G$ is a potential function, the following Poincar\'{e} inequality holds\begin{equation*}
		\begin{split}
			\norm{\G}{L^{{2}}}\leq C\norm{\nabla\G}{L^{{2}}},
		\end{split}
	\end{equation*}
from which and \eqref{tapoin} give rise  to
\begin{align*}
\mathcal{E}(t)\le{C\mathcal{D}(t)}.
\end{align*}
Then  we get
\begin{align*}
	\frac{d}{dt}{\mathcal{E}(t)}+c{\mathcal{E}(t)}\le 0.
\end{align*}
Solving this inequality yields
\begin{align}\label{boot4}
	{\mathcal{E}(t)}\le Ce^{-ct}.
\end{align}
Hence, we get
\begin{align}\label{boot5}
\int_0^t\big(\norm{ \aa(t')}{H^3}^2+\norm{ \u(t')}{H^3}^2+\norm{ \ta(t')}{H^3}^2\big)\,dt'
\le C.
\end{align}
Due to $$ \frac{1}{2}c_0\le\rho\le 2c_0^{-1},$$
we have
\begin{align*}
 \tilde{c}_0\le\frac{1}{1+a}\le \tilde{c}_0^{-1}.
\end{align*}
Hence, there holds
\begin{align*}
\frac12\norm{a}{H^{3}}^2-\frac{1}{2}\int_{\T^2}\frac{a}{1+a}(\la^{3}a)^2 \,dx\ge
C\norm{a}{H^{3}}^2,
\end{align*}
from which and Lemma \ref{gaojie}, we have
\begin{align}\label{boot6}
&\norm{(a,\u,\ta,m)}{H^{3}}^2
\leq \norm{(a_0,\u_0,\ta_0,m_0)}{H^{3}}^2\nn\\
&\quad+ C\int_0^t\big((1+\norm{a}{H^3}^2)(1+\norm{\u}{H^3})\norm{\u}{H^3}
       +\norm{(\aa,\u,\ta)}{H^3}^2\big)\norm{( a, \u,\ta, m)}{H^{3}}^2\,d\tau.
\end{align}
 From \eqref{boot4} and \eqref{boot5}, exploiting
the Gronwall inequality to \eqref{boot6}  implies that
\begin{align*}
\norm{(a,\u,\ta,m)}{H^{3}}^2
\le&C\norm{(a_0,\u_0,\ta_0,m_0)}{H^{3}}^2\\
&\quad\exp
\left\{C\int_0^T\big(
      (1+\norm{a}{H^3}^2)(1+\norm{\u}{H^3})\norm{\u}{H^3}
        +\norm{(\aa,\u,\ta)}{H^3}^2\big)\,d\tau\right\}\\
\le& C\varepsilon^2.
\end{align*}
Taking $\varepsilon$ small enough so that $C\varepsilon\le \delta/2$, we deduce from a continuity argument that the local solution can be extended as a global one in time. Consequently, we complete the proof of our main theorem.

\bigskip

\bigskip
 \section*{Acknowledgement.}
 {The first author of this work is supported by  the Guangdong Provincial Natural Science Foundation under grant 2022A1515011977. The third author of this work is supported by the China Postdoctoral Science Foundation under grant 2023TQ0309.}

\bigskip
\noindent{Conflict of Interest:} The authors declare that they have no conflict of interest.

\bigskip
\noindent{Data availability statement:}
 Data sharing not applicable to this article as no datasets were generated or analysed during the
current study.

\vskip .3in
\bibliographystyle{abbrv}

\end{document}